\documentclass[11pt, reqno]{amsart}

\newcommand{\Z}{\mathbb{Z}}
\newcommand{\R}{\mathbb{R}}

    \newcommand{\la}{\left\langle}
    \newcommand{\ra}{\right\rangle}

\usepackage[margin=3.6cm]{geometry}

\newcommand{\beq}{\begin{equation}}
\newcommand{\eeq}{\end{equation}}

\newcommand{\p}{\partial}
\newcommand{\re}{\operatorname{Re}}

\newcommand{\supp}{\operatorname{supp}}

\newcommand{\Op}{\operatorname{Op}}

\newcommand{\cO}{\mathcal{O}}

\def\O{\mathcal{O}}

\def\XXint#1#2#3{{
\setbox0=\hbox{$#1{#2#3}{\int}$}
\vcenter{\hbox{$#2#3$}}\kern-.5\wd0}}

\usepackage{amsmath, amsfonts, amsthm,amssymb,  mathrsfs,color}
\usepackage{tikz}

\newtheorem{theorem}{Theorem}[section]

\newtheorem{corollary}[theorem]{Corollary}
\newtheorem{lemma}[theorem]{Lemma}
\theoremstyle{definition}

\newtheorem*{remark}{Remark}

\begin{document}
\title{{Conformal Perturbations and Local Smoothing}}
\author{Hans Christianson}
\author{Dylan Muckerman}
\begin{abstract}
The purpose of this paper is to study the effect of conformal
perturbations on the local smoothing effect for the Schr\"odinger
equation on surfaces of revolution. The paper \cite{ChWu-lsm} studied
the Schr\"odinger equation on surfaces of revolution with one trapped
orbit. The dynamics near this trapping were unstable, but degenerately
so. Beginning from the metric $g$ from this paper, we consider the
perturbed metric $g_s = e^{sf}g$, where $f$ is a smooth, compactly
supported function. If $s$ is small enough and finitely many
derivatives of $f$ satisfy appropriate symbolic estimates, then we show that a
local smoothing estimate still holds.
\end{abstract}

\maketitle

\section{{Introduction and Statement of Results}}
Local smoothing estimates  for solutions to the Schr\"odinger equation are
estimates that use the infinite propagation speed to see
high-frequency wave packets leave a compact region faster than
low-frequency wave packets.  
In Euclidean space, the local smoothing result for the Schr\"odinger equation states that on average in time, and locally in space, solutions to the Schr\"odinger equation gain half a derivative compared to their initial data. More precisely, for every $T>0$ there exists $C_T>0$ such that if $u$ solves
\[
\begin{cases}
(D_t-\Delta)u = 0\\
\left.u\right|_{t=0} = u_0,
\end{cases}
\]
then
\[
\int_0^T \|\la r\ra^{-3/2}\p_ru\|^2+\|\la r\ra^{-1/2}r^{-1}\nabla_{S^{n-1}}u\|^2\,dt \le C_T\|u_0\|_{H^{1/2}}^2,
\]
for all $u_0 \in H^{1/2}$.   Here we have used polar coordinates with
$r$ the radial variable.  Note that the spatial weights are not sharp.

The idea of local smoothing was first studied by Kato \cite{Kato} in
the context of the KdV equation.  
Local smoothing for the linear Schr\"odinger equation and other
dispersive type equations was 
studied by Constantin-Saut \cite{ConSau}, Sj\"olin \cite{Sl}, Vega \cite{Vega}, and
Kato-Yajima \cite{KaYa-smooth}. Both \cite{Sl} and \cite{Vega} made use of this
inequality to prove that solutions of the Schr\"odinger equation
converge pointwise almost everywhere to their initial data as $t \to
0$.  We refer to 
 \cite{Tao-book} and \cite{ChWu-lsm} for simple proofs of this
 estimate.

\subsection{{Local smoothing in the presence of trapping}}

One perspective on the local smoothing effect is that it arises from
the dispersive nature of the Schr\"odinger equation. In particular,
high frequency parts of solutions to the Schr\"odinger equation have
higher velocity. By looking locally at solutions, we see ``less'' of
the high frequency part of our solution, and this is what is
responsible for the local smoothing.  In
Euclidean space, where the geodesics are straight lines, this is 
easy to visualize, and the $1/2$ derivative gain in the local
smoothing estimate makes this idea rigorous.  On the
other hand, on a compact manifold, wave packets have no where to
escape so no local smoothing is expected.

But many possibilities exist between Euclidean space and compact manifolds. According to our heuristic argument, the important property of Euclidean space is that every geodesic goes to infinity. In other words, there are no \emph{trapped} geodesics, where a trapped geodesic is a complete geodesic that remains in a compact set for all time.

The relationship between trapping and local smoothing was explored in \cite{Doi}: On asymptotically Euclidean manifolds, solutions to the Schr\"odinger equation exhibit $1/2$ of a derivative of local smoothing if and only if the manifold has no trapped geodesics.

The next question which arises is to what degree the local smoothing effect still holds when a trapped set exists.

The results in \cite{Bur-sm}, \cite{Chr-NC}, \cite{Chr-disp-1}, \cite{Chr-QMNC}, and \cite{Dat-sm}
showed that in the presence of non-degenerate hyperbolic trapping, for any $\epsilon > 0$, there is local smoothing of $1/2 - \epsilon$ derivatives for the Schr\"odinger equation.

\subsection{{Surfaces of Revolution}}

In \cite{ChWu-lsm}, local smoothing is studied on a family of surfaces of
revolution that have periodic geodesics which are unstable, but
degenerately so. In other words, the curvature vanishes to some degree
at the geodesic. The family of surfaces studied are given by rotating the curve
\[
A(x) = (1+x^{2m})^{1/2m},
\]
where $m \ge 2$ is an integer. The local smoothing effect is then
\[
\int_0^T \|\la x \ra^{-3/2}u\|^2_{H^1}\,dt \le C(\|\la D_\theta\ra^{m/(m+1)}u_0\|_{L^2}^2 + \|\la D_x\ra^{1/2} u_0\|_{L^2}^2).
\]
In other words, we gain the full $1/2$ of a derivative of local smoothing in the $x$ direction, but we only gain $1/(m+1)$ derivatives of local smoothing in the $\theta$ direction. Note that as the trapping becomes more stable, the local smoothing gained in the $\theta$ direction goes to $0$.

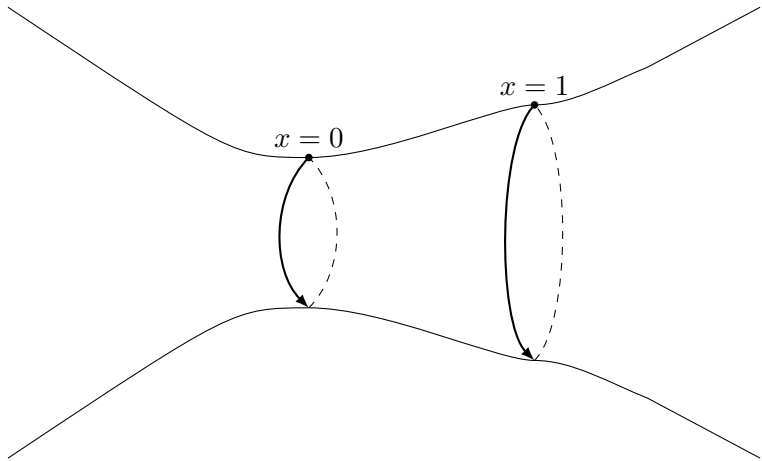
\begin{figure}[ht]
\centering
\begin{tikzpicture}
\draw (-4,3)..controls (-1,1)..(0,1)..controls (1,1) and (2.5,1.7)..
(3,1.7) ..controls(3.5,1.7) and (4,2) ..
(4.5,2.2) -- (6,3);

\fill(3,1.7) circle(.05) node[above]{$x = 1$};
\fill(0,1) circle(.05) node[above]{$x = 0$};

\draw (-4,-3)..controls (-1,-1)..(0,-1)..controls (1,-1) and (2.5,-1.7)..
(3,-1.7) ..controls(3.5,-1.7) and (4,-2) ..
(4.5,-2.2) -- (6,-3);

\draw[-latex,  thick](0,1) .. controls (-.5,.5) and (-.5,-.5)..(0,-1);
\draw[dashed](0,-1) ..controls(.5,-.5) and (.5,.5) ..(0,1);

\draw[-latex,  thick](3,1.7) .. controls (2.5,1.2) and (2.5,-1.2)..(3,-1.7);
\draw[dashed](3,-1.7) ..controls(3.5,-1.2) and (3.5,1.2) ..(3,1.7);

\end{tikzpicture}
\caption{\label{F:trapping} A piece of the manifold with trapped
  geodesics at $x=0$ and $x=1$.  The geodesic at $x = 0$ is
  degenerately unstable (studied in \cite{ChWu-lsm}) and the geodesic
  at $x = 1$ is of inflection-transmission type (studied in
  \cite{ChMe-lsm}).  In this paper we consider a conformal
  perturbation of the manifold studied in \cite{ChWu-lsm}, without any
inflection-transmission trapping.}
\end{figure}

In \cite{ChMe-lsm}, a similar result is proven for a related class of
surfaces of revolution with inflection-transmission type trapping.   
It should be noted that the results of \cite{ChWu-lsm} and \cite{ChMe-lsm} are sharp and show that no better (lower) power of $\la D_\theta\ra$ is possible.

Finally, \cite{Chr-inf-deg} gives details of the connection between resolvent estimates for the Laplacian and local smoothing, and a detailed exposition of how the results obtained in \cite{ChWu-lsm} and \cite{ChMe-lsm} can be combined via ``gluing'' to prove local smoothing results for a wide variety of warped product manifolds.

Similar results are also available for localized energy estimates for the wave equation on surfaces of revolution with degenerate trapping in \cite{BCMP}.

\subsection{{Conformal perturbations of surfaces of revolution}}
The previous results mentioned above essentially complete the study of local smoothing for the Schr\"odinger equation on surfaces of revolution (and warped product manifolds in general). All of these results are essentially 1 dimensional, thanks to the decomposition into Fourier modes. 
We study local smoothing on a family of surfaces which are conformal
perturbations of the surfaces studied in \cite{ChWu-lsm}.

Recall that a surface of revolution is the manifold $M = \R_x \times \R_{\theta}/2\pi\Z$ endowed with the metric
\[
g_0 = dx^2 + A^{2}(x)d\theta^2,
\]
where $A > 0$. We consider conformal perturbations of this metric in which the metric is of the form
\[
g_s = e^{sf(x,\theta)}g_0,
\]
where $f(x,\theta)$ is a smooth function, compactly supported in $x$. Note that
\[
\Delta_{g_s} = e^{-sf}\Delta_{g_0}.
\]

We will work with the function $A$ given in \cite{ChWu-lsm}. Note that if $f$ depends only on $x$, then after the perturbation our surface retains its rotational symmetry and so is still a surface of revolution, though it is impractical to write down its metric explicitly in the standard form for surfaces of revolution.

If our perturbation function $f$ has appropriate conditions placed on it, one expects that it will have little effect on the dynamics near the trapped set and thus little effect on the local smoothing. In fact, it is reasonable to expect that the perturbation could make the dynamics less stable and thus lead to greater local smoothing, though this is beyond our scope.

\begin{theorem}
Let $\epsilon > 0$ and let $M = \R_x \times \R_{\theta}/2\pi \Z$ endowed with the metric
\[
g = e^{sf(x,\theta)}(dx^2+A^{2}(x)d\theta^2),
\]
where 
\[
A(x) = (1+x^{2m})^{1/2m},
\]
$m \in \Z$, $m \geq 2$, 
and $f \in C^\infty(M)$ is compactly supported in $x$ and satisfies
\[
|\p_x^j\p_\theta^k f| \le C|x|^{2m-1}
\]
for $x$ small and $j, k \le N$ for sufficiently large $N = N(m, \epsilon)$ where $j + k \ge 1$. 
Let
\[
r = \frac{m}{m+1} + \epsilon.
\]

Then for $s>0$ sufficiently small, there exists $C_{T} > 0$ such that 
\[
\int_0^T (\|\la x \ra^{-1}\p_x u\|^2 + \| \la x \ra^{-3/2} \p_\theta u\|^2)\,dt \le C_{T}
\|u_0\|^2_{H^{r}}
\]
for all $u$ solving the Schr\"odinger equation
\[
\begin{cases}
(D_t-\Delta_g)u = 0\\
u|_{t=0} = u_0 \in \mathcal{S}.
\end{cases}
\]
\end{theorem}

\begin{remark}
  We have assumed that the initial data $u_0 \in \mathcal{S}$ is
  Schwartz class to avoid any issues with integration by parts. 

In the unperturbed case, there is a gain of
\[
\frac{1}{m+1}
\]
derivatives, whereas in our case there is the gain of
\[
\frac{1}{m+1} - \epsilon
\]
derivatives. 
This is because we have chosen to avoid the marginal calculus used in \cite{ChWu-lsm}, in order to ensure gains (in terms of $\theta$ derivatives) in symbol expansions, so that the many extra terms introduced by the factor $e^{-sf}$ are easier to control.
\end{remark}

\begin{remark}
Note that we do not require any bound on $f$ itself, only on its
derivatives.  We have stated the theorem in terms of $s>0$, which is a
convenience just to avoid excessive notation of $|s|$ every time we estimate.  We have made no assumptions on the
sign of the function $f$.
\end{remark}

\begin{remark}
The intuitive reason for our condition on derivatives of $f$ is that in general the degenerate trapping found in the unperturbed manifold is unstable under perturbation, and could potentially be perturbed into much worse trapping, for which the result would not hold.

We also note that non-degenerate hyperbolic trapping is stable under perturbation, so there is no corresponding issue in that situation. 
\end{remark}

\section{{Notation and Preliminary Material}}

\subsection{Notations and Conventions}
We will use $C$ to denote a large constant which may change from line
to line. We will similarly use $c$ to denote a small positive constant
which may change from line to line.  We use the bracket notation:
\[
\la x \ra^s = (1 + x^2)^{s/2},
\]
for any $s \in \R$, which is a smooth positive function and $\la x \ra^s \sim |x|^s$ at infinity.  

\subsubsection{{Pseudodifferential Operators}}
Our outline of pseudodifferential operators will follow the
presentation of \cite{zw-book}, \cite{Tay-pdo}, and \cite{Tay2}.  For
the use of pseudodifferential operators on a circle with discrete
frequency parameter, we follow \cite{PDO-on-torus}.


We will work with the \emph{symbol classes} $S^m_{\rho}$, $\rho \ge 0$ originally defined in \cite{HoHyp}, given by
\[
S^m_{\rho} = \{a \in C^\infty(\R \times \R \times S^1 \times \Z)\ :\ |\p_\xi^\alpha \p^\beta_x \p_\theta^\gamma \p^\delta_\eta a| \le C_{\alpha,\beta,\delta,\gamma}\la \xi \ra^{m-|\alpha|\rho}\la \eta \ra^{-|\delta|\rho}\},
\]
where $\p_\eta$ denotes a difference operator in $\eta$, defined by
\[
\p_\eta a(\cdot , \eta) = a(\cdot, \eta +1) - a(\cdot, \eta).
\]
 In particular, we will work with a symbol supported only where $|\xi|
 \le C|\eta|$, allowing us to transfer decay in $|\eta|$ to decay in
 $|\xi|$.  We also point out that $\rho = 0$ is the ``marginal case''
 where derivatives gain nothing.  The marginal calculus was used in
 \cite{ChWu-lsm} to obtain a sharp result, but in this paper we will
 always have $\rho >0$.  It is unclear if this is necessary, or if the
 marginal calculus can be used in the present context.

Define
\[
 a^wu = \frac{1}{(2\pi)^2}\int_{\R^2}\int_{S^1}\sum_{\eta} e^{i\la x-\tilde{x},\xi\ra + i\la \theta-\tilde\theta, \eta\ra} a\left(\frac{x+\tilde{x}}{2}, \frac{\theta + \tilde\theta}{2}, \xi,\eta\right)u(\tilde{x}, \tilde{\theta})\,d\tilde{\theta}d\tilde{x}d\xi
\]
The operator $a^w$ is a \emph{pseudodifferential operator} obtained from taking the \emph{Weyl quantization} of $a$. It should be noted that the Weyl quantization is just one choice of many quantizations. The function $a$ is said to be the \emph{symbol} of the operator.

\subsubsection{Symbol calculus}

We review a few essential theorems of the symbol calculus.

\begin{theorem}[Calderon-Vaillancourt Theorem]
If $a \in S^0_0$ then the operator $a^w$ is bounded as an operator from $L^2$ to $L^2$.
\end{theorem}
This theorem is originally due to \cite{CV}. See Theorem 4.23 in \cite{zw-book} for another proof.

In fact, a more general theorem holds.
\begin{theorem}
If $a \in S^m_0$ then the operator $a^w(x,D)$ is bounded as an operator from $H^{s+m}$ to $H^{s}$.
\end{theorem}

Quantization does not commute with composition. That is to say, the composition of two pseudodifferential operators is not the quantization of the product of their symbols. In fact, it is not immediately obvious that the composition of two pseudodifferential operators is a pseudodifferential operator. In fact, the following theorem holds.

\begin{theorem}[Theorem 4.18 in \cite{zw-book}]\label{theorem-symb-exp}\label{symb-1}
Let $a \in S^m_{\rho}$, $b \in S^{\tilde{m}}_{\rho}$. Let
\[
A(D) = \frac{1}{2}(\la D_\xi, D_y\ra - \la D_x,D_\eta\ra).
\]
Then
\[
a^w(x,D) \circ b^w(x,D) = c^w(x,D)
\]
for
\[
c = a\#b := \sum_{k=0}^N \frac{i^k}{k!}A(D)^ka(x,\xi)b(y,\eta) \biggr|_{x = y, \xi = \eta} + r,
\]
where $r$ is a symbol in $S^{m+\tilde{m}-N\rho}_{\rho}$. Furthermore, the symbol $c$ is in the class $S_\rho^{m+\tilde{m}}$.
\end{theorem}

In particular, we have the following Corollary.
\begin{corollary}
Let $a \in S^m_{\rho}$, $b \in S^{\tilde{m}}_{\rho}$. Then
\[
a\#b = ab + \frac{1}{2i}\{a,b\} + r,
\]
where $r \in S^{m+\tilde{m} - 2\rho}_{\rho}$.
\end{corollary}

This can be seen from the symbol expansion for the commutator of $a^w$ and $b^w$ using Theorem \ref{theorem-symb-exp}. Due to the symmetry of the Weyl quantization, the following holds.
\begin{corollary}
Let $a \in S^m_{\rho}$, $b \in S^{\tilde{m}}_{\rho}$. Then the commutator
\[
[a^w(x,D), b^w(x,D)] = c^w(x,D),
\]
where
\[
c = \frac{1}{i}\{a,b\} + r,
\]
and $r \in S^{m + \tilde{m} - 3\rho}_{\rho}$.
\end{corollary}
Note that we gain $3$ in the symbol class of the remainder term, rather than the gain of $2$ we may naively expect. See Theorem 4.12 in \cite{zw-book}.

Another useful feature of the Weyl quantization is the following theorem. 
\begin{theorem}
Let $a$ be a real symbol. Then the operator $a^w$ is essentially self-adjoint.
\end{theorem}

A final result we require is the G\aa rding inequality.
\begin{theorem}
Let $a \in S^m_{\rho}$ with $0 \le \rho \le 1$ and suppose
\[
\re a \ge C|(\xi,\eta)|^m
\]
for $|(\xi,\eta)|$ large. Then for any $m \in \R$ there exist $C_1, C_2$ such that for all $u \in H^{m/2}$,
\[
\re \la a^wu,u\ra \ge C_1\|u\|_{H^{m/2}}^2 - C_2\|u\|_{H^{(m-\rho)/2}}^2.
\]
\end{theorem}
See Chapter 7, Theorem 6.1 in \cite{Tay2} for a proof.

\section{{Positive Commutator}}

The Laplacian $\Delta_{g_0}$ on the unperturbed metric is given by
\[
\Delta_{g_0} = \p_x^2+A^{-2}(x)\p_\theta^2+A^{-1}(x)A'(x)\p_x.
\]
Define
\[ L_1: L^2(X, dVol) \to L^2(X,dxd\theta)\]
by
\[ L_1u(x,\theta) = A^{1/2}(x)u(x,\theta)\]
and define
\[
L_2: L^2(e^{sf}dxd\theta) \to L^2(dxd\theta)
\]
by
\[
L_2u(x,\theta) = e^{sf/2}u.
\]
Let $\tilde{\Delta} = L_2L_1\Delta_g L_1^{-1}L_2^{-1}$. Let
\[
V_1(x) = \frac{1}{2}A''(x)A^{-1}(x)-\frac{1}{4}(A'(x))^2A^{-2}(x)
\]
We compute $\tilde{\Delta}$ explicitly and find
\begin{align*}
\tilde{\Delta} u &= e^{-sf/2} \left( \p_x^2 + A^{-2} \p_\theta^2 - V_1(x) \right) e^{-sf/2} \\
&= e^{-sf}\left(\p_x^2+ A^{-2}\p_\theta^2\right) \\
& \quad +e^{-sf}(-sf_x\p_x - A^{-2}sf_{\theta}\p_\theta -
(s/2)f_{xx}+((s/2)f_x)^2
)\\
& \quad +e^{-sf}(
- A^{-2}(s/2)f_{\theta\theta} + A^{-2}((s/2)f_\theta)^2)\\
&\quad - e^{-sf}V_1(x).
\end{align*}

We note that
\begin{align*}
(e^{-sf}(\xi^2 + A^{-2}(x)\eta^2 + V_1(x)))^w &= -\tilde\Delta
\end{align*}

Let
\[
Q = (e^{-sf}(\xi^2+A^{-2}\eta^2))^w
\]
and
\[
R = -e^{-sf}(-sf_x\p_x - A^{-2}sf_{\theta}\p_\theta -
(s/2)f_{xx}+((s/2)f_x)^2 - A^{-2}(s/2) f_{\theta\theta} + A^{-2}((s/2)f_\theta)^2)
\]
so that 
\[
Q = - e^{-sf}(\p_x^2 + A^{-2}\p_\theta^2) + R.
\]
Then $Q$ is essentially self-adjoint and $R$ consists of the lower order parts of the operator.

Below we will commute with an operator $B$ involving only $1$
derivative. Commuting $B$ and $e^{-sf}V_1(x)$ will produce a bounded
function and no derivatives, or in other words an $L^2$ bounded
operator.  This can then easily be absorbed into the upper bound of $\|u_0\|_{H^{1/2}}^2$, as will be done with many other remainder terms below. Thus proving the result for $Q$ will prove the result for $\tilde{\Delta}$. Conjugating back then proves the result for $\Delta_g$. For this reason, we will leave out $V_1(x)$ in the computations below and work with $Q$.

We begin by making the same positive commutator argument as in \cite{ChWu-lsm}. By commuting the operator we are interested in, $Q$, with an appropriate operator $B$ we are able to prove the local smoothing estimate away from $x = 0$.

For our commutant we choose
\[
B = \arctan(x)\p_x.
\]
We begin by commuting the two operators to find
\begin{align}\label{com-exp}
[Q,B] &= -e^{-sf}(\p_x^2+A^{-2}\p_\theta^2)[\arctan(x)\p_x] \\\notag
&\quad+\arctan(x)\p_x[e^{-sf}(\p_x^2+A^{-2}\p_\theta^2)] + [R,B]\\\notag
&=-e^{-sf}\biggr[2\la x\ra^{-2}\p_x^2 - \frac{2x}{(1+x^2)^{2}}\p_x +sf_x\arctan(x)(\p_x^2+A^{-2}\p_\theta^2)\\\notag
&\quad +\arctan(x)2A'A^{-3}\p_\theta^2\biggr] + [R,B].
\end{align}

Now that we are done with the preliminary computations, we begin the argument proper by assuming that $u$ satisfies the Schr\"odinger equation
\[
\begin{cases}
(D_t+Q)u = 0,\\
u(0,x,\theta) = u_0 \in \mathcal{S}.
\end{cases}
\]

Using that, we write down the following expression which equals $0$:
\[
0 = \int_0^T \la B(D_t+Q)u,u\ra - \la Bu,(D_t+Q)u\ra\,dt.
\]
In order to make our commutator term appear, we next need to integrate by parts in the second term and obtain
\[
0 = \int_0^T\la B(D_t+Q)u,u\ra - \la (D_t+Q)Bu,u\ra\,dt +\left.i\la Bu,u\ra\right|_0^T.
\]
We combine the terms involving $D_t$, $Q$, and $B$. This results in
\[
0 = \int_0^T\la B(D_t+Q) - (D_t+Q)Bu,u\ra\,dt + \left.i\la Bu,u\ra\right|_{0}^T.\]

Finally, we note that these combined terms are precisely the commutator we computed above, and we end up with the equation
\begin{equation}
  \label{E:comm-101}
 \int_0^T\la [Q,B]u,u\ra = \left.i\la Bu,u\ra\right|_0^T.
\end{equation}

Next we write out the commutator in \eqref{E:comm-101} and move the
largest, 
highest order
terms  to the left hand side.   That is, we move the terms with zero
or one derivative to the right hand side, as well as terms multiplied
by $s$.  
This results in the equation
\begin{align}
\int_0^T &\la -e^{-sf} 2\la x\ra^{-2}\p_x^2u,u\ra - \la
e^{-sf}\arctan(x)2A'A^{-3}\p_\theta^2u,u\ra\,dt\notag \\
&= - \int_0^T \la e^{-sf}\left[\frac{2x}{(1+x^2)^{2}}\p_x
  +sf_x\arctan(x)(\p_x^2+A^{-2}\p_\theta^2) + [R,B]\right]u,u\ra\,dt
\notag \\
&\quad + \left.i\la Bu,u\ra\right|_0^T. \label{E:comm-102}
\end{align}

We begin by working on the left hand side of \eqref{E:comm-102}.
Starting with the
first term on the left hand side of \eqref{E:comm-102} involving derivatives of $x$, we first integrate by parts to find
\begin{equation}
  \label{E:comm-103}
-\la e^{-sf}2\la x \ra^{-2}\p_x^2u,u\ra = \la \p_xu,(\p_x[2e^{-sf}\la x \ra^{-2}u]\ra.
\end{equation}
Next we use the product rule to find that \eqref{E:comm-103} equals
\[ \|e^{-sf/2}\la x\ra^{-1}\p_x u\|^2 + \la \p_x u, \left(-2sf_xe^{-sf}\la x \ra^{-2} - \frac{4xe^{-sf}}{(1+x^2)^{2}}\right)u\ra.
\]
The first term here is the highest order giving an $H^1$ norm.  
We move the second term to the right hand side of \eqref{E:comm-102} and bound it above. First we note that the function
\[
-2sf_xe^{-sf}\la x \ra^{-2} - \frac{4xe^{-sf}}{(1+x^2)^2}
\]
and all of its derivatives are bounded. We can then split the $\p_x$
across both parts of the inner product and obtain an upper bound of
$C\|u\|_{H^{1/2}}^2$ as follows: First we apply the operator $\la
D_x\ra^{1/2}\la D_x\ra^{-1/2}$, and then we use integration by parts.

This term then equals
\begin{align}
  & \la \p_x u, \left(-2sf_xe^{-sf}\la x \ra^{-2} -
  \frac{4xe^{-sf}}{(1+x^2)^{2}}\right)u\ra \notag \\
  & \quad = 
\la \la D_x\ra^{-1/2}\p_x u, \left(\la D_x\ra^{1/2}((-2sf_xe^{-sf}\la
x \ra^{-2} - \frac{4xe^{-sf}}{(1+x^2)^{2}}\right)u)\ra. \label{E:comm-104}
\end{align}
Using the Cauchy-Schwarz inequality we are able to bound \eqref{E:comm-104} from above by
\[
C\|\la D_x\ra^{-1/2}\p_x u\|_{L^2} \left\|\left(\la
D_x\ra^{1/2}((-2sf_xe^{-sf}\la x \ra^{-2} -
\frac{4xe^{-sf}}{(1+x^2)^{2}}\right)u)\right\|_{L^2} \leq C\|u\|^2_{H^{1/2}}.
\]

Next we move on to the term in \eqref{E:comm-102} involving derivatives of $\theta$ and proceed similarly. We have
\begin{align}
-&\la e^{-sf}\arctan(x)A'A^{-3}\p_\theta^2u,u\ra \notag \\
&\quad =\la e^{-sf}\arctan(x)x^{2m-1}(1+x^{2m})^{-1/m - 1}\p_\theta
  u,\p_\theta u\ra \notag \\
&\quad \quad + s\la f_\theta e^{-sf}\arctan(x)x^{2m-1}(1+x^{2m})^{-1/m
    - 1}\p_\theta u,u\ra. \label{E:comm-105}
\end{align}
The last term in \eqref{E:comm-105} involving only a single $\theta$
derivative is controlled by $C\|u\|_{H^{1/2}}^2$, just as we
did for the terms involving only a single $x$ derivative in \eqref{E:comm-104}.
In fact, returning to \eqref{E:comm-102}, we can use energy estimates
to similarly estimate each first order term on the right hand side by
$C_T \| u_0 \|^2_{H^{1/2}}$.  Here we emphasize that the constant
$C_T$ does depend on $T$.

Thus far we have proven the inequality
\begin{align}\label{LHS-POS-COM-1}
\int_0^T &\|e^{-sf/2}\la x \ra^{-1}\p_x u\|_{L^2}^2 + \la e^{-sf}\arctan(x)x^{2m-1}(1+x^{2m})^{-1/m-1}\p_\theta u,\p_\theta u\ra\,dt\\\notag
&\leq \int_0^T \left| \la sf_x\arctan(x)(\p_x^2+A^{-2}\p_\theta^2) u,u\ra
\right| \,dt  + C_T \|u_0 \|^2_{H^{1/2}}.
\end{align}

The first term on the left hand side of \eqref{LHS-POS-COM-1}  is
already written as a norm. For the second term, we need to do a bit of
work before it can be bounded below by a norm.  Note that 
\[
\la e^{-sf} |x|^{2m}\la x \ra^{-2m-3}\p_\theta u,\p_\theta u\ra \le  C\la e^{-sf}\arctan(x)x^{2m-1}(1+x^{2m})^{-1/m-1}\p_\theta u,\p_\theta u\ra,
\]
So we may bound the left hand side of (\ref{LHS-POS-COM-1}) below by 
\begin{equation}\label{LHS-POS-COM-1.5}
c\int_0^T \|e^{-sf} \la x \ra^{-1}\p_x u\|_{L^2}^2 + \|e^{-sf}|x|^{m}\la x\ra^{-m-3/2}\p_\theta u\|_{L^2}^2\,dt,
\end{equation}
for some $c > 0$.
Finally, we can drop the factors of $e^{-sf}$ by using the fact that
$f$ is compactly supported and hence $e^{-sf}$ is bounded below by
some $c > 0$. Thus the lower bound of the left hand side of
\eqref{LHS-POS-COM-1} is 
\begin{equation}\label{LHS-POS-COM-2}
c\int_0^T \|\la x \ra^{-1}\p_x u\|_{L^2}^2 + \||x|^{m}\la x\ra^{-m-3/2}\p_\theta u\|_{L^2}^2\,dt.
\end{equation}

So far we have shown
\begin{align}
& c\int_0^T \|\la x \ra^{-1}\p_x u\|_{L^2}^2 + \||x|^{m}\la
  x\ra^{-m-3/2}\p_\theta u\|_{L^2}^2\,dt \notag \\
  & \quad \leq C_T\| u_0 \|^2_{H^{1/2}} + \int_0^T \left| \la sf_x\arctan(x)(\p_x^2+A^{-2}\p_\theta^2) u,u\ra
  \right| \,dt. \label{E:comm-106}
  \end{align}
The strategy for dealing with the terms with two derivatives on the
right hand side of \eqref{E:comm-106} is to
make use of the fact that $s$ is small to absorb them into the left
hand side of \eqref{E:comm-106}.   
Integrating by parts and using energy estimates on the lower order terms, we have
\begin{align}
 \int_0^T & \left| \la sf_x\arctan(x)\p_x^2 u,u\ra
 \right| \,dt \notag \\
 & = 
\int_0^T \left|\la se^{-sf}\left(-s(f_x)^2\arctan(x) + f_{xx}
  \arctan(x) 
  \right)u,\p_x
  u\ra\right|\,dt \notag \\
& \quad + \int_0^T \left|\la se^{-sf}\left(
  f_x\la x \ra^{-2} + f_x\arctan(x)\p_x\right)u,\p_x
  u\ra\right|\,dt \notag \\
  & \leq C_T \| u_0 \|^2_{H^{1/2}} +  \int_0^T \left|\la se^{-sf}
  f_x\arctan(x)\p_x u,\p_x u\ra\ \right|dt. \label{E:comm-error-101}
  \end{align}
By making use of the fact that $f$ is compactly supported, we can then
bound \eqref{E:comm-error-101} above by
\[ C_T\|u_0\|_{H^{1/2}}^2 + Cs\int_0^T \|\la x \ra^{-1} \p_x u\|^2_{L^2}\]
The second of these terms may be moved to the left hand side of \eqref{E:comm-106}, provided that $s$ is sufficiently small.


Similarly, for the term with two $\theta$ derivatives on the right
hand side of \eqref{E:comm-106}, we compute
\begin{align*}
\int_0^T &\left|\la e^{-sf}sf_x\arctan(x)A^{-2}\p_\theta^2)u,u\ra\right|\,dt \\
&= s\int_0^T \left|\la \p_\theta(e^{-sf}f_x\arctan(x)A^{-2}u),\p_\theta u\ra\right|\,dt\\
&=s\int_0^T \left|\la e^{-sf}\arctan(x)A^{-2}(-sf_{\theta}f_x+f_{x\theta} +f_x\p_\theta)u,\p_\theta u\ra\right|\,dt\\
&\le C_T\|u_0\|_{H^{1/2}}^2 + \int_0^T s\left|\la e^{-sf}f_x\arctan(x)A^{-2}\p_\theta u,\p_\theta u\ra\right|\,dt.
\end{align*}
Recall we have assumed that 
\[|f_x| \le C|x|^{2m-1}
\]
in a neighborhood of $x=0$. Then using also the fact that $f$ is
compactly supported and $\arctan(0) = 0$, we have
\[
|sf_x\arctan(x)| \le Cs|x|^{2m}\la x \ra^{-2m-3},
\]
and thus
\[
\int_0^T s\left|\la e^{-sf}f_x\arctan(x)A^{-2}\p_\theta u,\p_\theta u\ra\right|\,dt \le Cs\int_0^T \||x|^{m}\la x \ra^{-m-3/2}\p_\theta u\|^2_{L^2}\,dt.
\]
By choosing $s$ sufficiently small we may absorb this into the left
hand side of \eqref{E:comm-106}.
We thus have the estimate
\begin{equation}\label{POS-COM}
\int_0^T \|\la x \ra^{-1}\p_x u\|_{L^2}^2 + \||x|^{m}\la x\ra^{-m-3/2}\p_\theta u\|_{L^2}^2\,dt \le  C_{T}\|u_0\|_{H^{1/2}}^2
\end{equation}

This estimate shows that the local smoothing is perfect away from  $x = 0$, and that we have perfect local smoothing in the $x$ direction. Next we will work on the local smoothing in the $\theta$ direction and near $x = 0$.

\section{{Estimating in the Frequency Domain}}
Our plan is to split the function $u$ up based on whether $|D_x|$ or $\la D_\theta\ra$ is larger, writing $u = u_1 + u_2$, so that $u_2$ satisfies the estimate
\[
\|\la D_\theta\ra u_2\|_{L^2} \lesssim \|\p_x u_2\|_{L^2}.
\]
We give an outline of the proof before proceeding with the proof. First we repeat the above argument using $u_2$ in place of $u$. Because $u_2$ is only approximately a solution to the Schr\"odinger equation, there will be additional error terms. The lower bound of
\[
\int_0^T \|\la x \ra^{-1}\p_x u_2\|_{L^2}^2\, dt.
\]
can be bounded from below by
\[
\int_0^T \|\la x \ra^{-1}\la D_\theta\ra u_2\|_{L^2}^2\,dt
\]
This gives us a lower bound in the $\theta$ direction away from $x =
0$. However, it is only for $u_2$, and the upper bound will involve a
term other than $\|u_0\|_{H^{1/2}}^2$, due to the fact that $u_2$ does
not solve the Schr\"odinger equation.
We will  reduce the problem to finding an appropriate estimate for $u_1$, which will be the subject of the remaining sections.

Let $\psi(\tau)$ be a bump function with $\psi(\tau) = 0$ for $|\tau| > 2$ and $\psi(\tau) = 1$ for $|\tau| < 1$. We define the operator $\psi(D_x/\la D_\theta\ra)$ as a Fourier multiplier. Let $\hat{u}(t,\xi,\eta)$ denote the Fourier transform of $u$ in $x$ and $\theta$. Because $\theta \in S^1$, $\eta$ takes integer values. Let $\mathcal{F}$ denote also this Fourier transform:
\[
(\mathcal{F}u)(\xi,\eta) = \int_{\R}\int_{S^1} e^{-ix\xi}e^{-i\theta \eta}u(x,\theta)\,d\theta dx.
\]
Let $\mathcal{F}^{-1}$ denote the inverse. Note that $\mathcal{F}^{-1}$ involves an integral in $\xi$ but a sum in $\eta$:
\[
(\mathcal{F}^{-1}v)(x,\theta) = \frac{1}{4\pi^2}\int_{\R}\sum_{\eta \in \Z} e^{ix\xi}e^{i\theta\eta}v(\xi,\eta)\,d\xi
\]
We then define
\[
\psi(D_x/\la D_\theta\ra)u = \mathcal{F}^{-1}(\psi(\xi/\la \eta\ra) \hat{u}).
\]
Again suppose $u$ solves
\[
\begin{cases}
  (D_t+Q)u = 0,\\
u(0,x,\theta) = u_0 \in \mathcal{S}.
\end{cases}
\]

We will consider $u_1 = \psi(D_x/\la D_\theta\ra)u$ and $u_2 = (1-\psi(D_x/\la D_\theta\ra))u$.
While $u_2$ is not a solution to the Schr\"odinger equation, we will show that it is close enough to a solution for our purposes. We have
\begin{align}
(D_t+Q)u_2 &= (D_t+Q)[(1-\psi(D_x/\la D_\theta\ra))u] \notag \\
&= (1-\psi(D_x/\la D_\theta\ra))(D_t+Q)u - [Q,\psi(D_x/\la
    D_\theta\ra)]u \notag \\
&=-[Q,\psi(D_x/\la D_\theta\ra)]u. \label{E:Q-comm-101}
\end{align}
We pause briefly to point out at this point that the commutator with
$\psi$ above gains both regularity {\it and} decay in the radial
variable $x$.  We use that shortly.

Letting $B = \arctan(x)\p_x$ as above we repeat the positive commutator argument from above. We begin by simply expanding the commutator to find
\[
\int_0^T \la [Q,B]u_2,u_2\ra\,dt = \int_0^T \la QBu_2,u_2\ra - \la BQu_2,u_2\ra\,dt.
\]
Next we want to have both $Q$'s be applied to $u_2$ so that we can use what we know about $u_2$ and the Schr\"odinger equation. 
We then proceed as in the calculations following (\ref{com-exp}) to find
\begin{align}
\int_0^T \la [Q,B]u_2,u_2\ra\,dt &=\int_0^T\biggr[\la Bu_2,(D_t+Q)u_2\ra
 -\la B(D_t+Q)u_2,u_2\ra\biggr]\,dt \notag \\
 &\quad +i\la Bu_2,u_2\ra\biggr|_{t=0}^T.  \label{E:comm-u2}
\end{align}
Our lower bound will come from the left hand side of the equality, while the right hand side will need to be bounded from above, in a manner similar to the preceding section.

Next we consider
\[
\int_0^T \la Bu_2,(D_t+Q)u_2\ra\,dt.
\]
We write this as
\begin{align*}
& \left|\int_0^T \la\la x \ra^{-1}Bu_2, \la x \ra(D_t+Q)u_2\ra\,dt
  \right| \\
  & \quad \le C\int_0^T(\|\la x \ra^{-1} Bu_2\|_{L^2}^2 + \|\la x \ra(D_t+Q)u_2\|_{L^2}^2)\,dt.
\end{align*}
Note that $[\psi(D_x/\la D_\theta\ra),\p_x]=0$.  Using the
inequality \eqref{POS-COM}  we proved in the previous section, we then know
\begin{align*}
  & \int_0^T\|\la x \ra^{-1}Bu_2\|^2\,dt \\
  & \quad \le C\int_0^T \|\la x\ra^{-1}\p_x u\|^2\,dt + C\|u\|_{L^2}^2 \le C_T\|u_0\|_{H^{1/2}}^2.
\end{align*}

Recall that $u_2$ satisfies  \eqref{E:Q-comm-101}, so that  
\begin{align}
  & (D_t+Q)u_2 \notag \\
  &= -[Q, \psi(D_x/\la D_\theta\ra)]u \notag \\
&= -[e^{-sf}(D_x^2+A^{-2}D_\theta^2) + R, \psi(D_x/\la D_\theta\ra)]u
  \notag \\
&=-[e^{-sf},\psi(D_x/\la D_\theta\ra)](D_x^2 + A^{-2}D_\theta^2)u -
  e^{-sf}[D_x^2 + A^{-2}D_\theta^2, \psi(D_x/\la D_\theta\ra)]u\notag \\
&\quad + [R, \psi(D_x/\la D_\theta\ra)]u. \label{E:Q-comm-102}
\end{align}
To estimate  the first term on the right hand side of
\eqref{E:Q-comm-102}, we make note of the commutator terms.   Because $f$ has compact support in $x$, we have decay
in $x$ as quickly as we like. Because of the $\psi'$ appearing in the
commutator, we will be working in the region where $D_x \sim \la
D_\theta\ra$, and we will gain a power of $D_x$ or $\la D_\theta \ra$,
whichever is more useful. Thus
\[
\|\la x \ra[e^{-sf},\psi(D_x/\la D_\theta\ra)](D_x^2+A^{-2}D_{\theta}^2)u\| \le C\|\la x \ra^{-1} D_xu\|.
\]
We may then bound $\int_0^T\|\la x \ra^{-1} D_xu\|^2\,dt$ by $C_T\|u_0\|_{H^{1/2}}^2$ as we did before.

Next we note that
\[
[D_x^2 + A^{-2}D_\theta^2, \psi(D_x/\la D_\theta\ra)] = [A^{-2}(x), \psi(D_x/\la D_\theta\ra)]D_\theta^2.
\]
We then have
\[
\la x \ra[A^{-2}(x),\psi(D_x/\la D_\theta\ra)]D_\theta^2 = L\la x\ra^{-2}D_\theta \tilde{\psi}(D_x/\la D_\theta\ra),
\]
where $L$ is $L^2$-bounded and $\tilde{\psi} \in C_0^\infty$ equals $1$ on $\supp \psi$. Then
\begin{align*}
\int_0^T \|\la x \ra[A^{-2}(x),\psi(D_x/\la D_\theta\ra)]D_\theta^2u\|^2\,dt &= \int_0^T \|L\la x\ra^{-2}D_\theta \tilde{\psi}(D_x/\la D_\theta\ra)u\|^2\,dt\\
&\le C\int_0^T\|\la x \ra^{-2}D_\theta \tilde{\psi}(D_x/\la D_\theta\ra)u\|^2\,dt.
\end{align*}
Controlling this will be the subject of the next section.

The term 
\[
\int_0^T \la B(D_t+Q)u_2,u_2\ra\,dt = \int_0^T \la (D_t+Q)u_2, B^*u_2\ra\,dt
\]
 from \eqref{E:comm-u2} is controlled in exactly the same fashion.
Thus far we have shown
\begin{align*}
\left| \int_0^T \la[Q,B]u_2,u_2\ra\,dt \right| \le CT\|u_0\|_{H^{1/2}}^2 + \int_0^T\|\la x \ra^{-2}D_\theta \tilde{\psi}(D_x/\la D_\theta\ra)u\|^2\,dt.
\end{align*}

Next we use our expansion of $[Q,B]$ given in (\ref{com-exp}) above:
\begin{align*}
[Q,B] &= -e^{-sf}\biggr[2\la x \ra^{-2}\p_x^2 - \frac{2x}{(1+x^2)^2}\p_x - sf_x\arctan(x)(\p_x^2 + A^{-2}(x)\p_\theta^2)\\
&\quad + \arctan(x)2A'A^{-3}\p_\theta^2\biggr] + [R, B].
\end{align*}
We will examine
\begin{align}
\int_0^T & \la[Q,B]u_2, u_2 \ra dt \notag \\
& = \int_0^T \la -e^{-sf}(2\la x \ra^{-2}\p_x^2u_2, u_2 \ra dt \notag
\\
& \quad
+ \int_0^T \la e^{-sf} \frac{2x}{(1+x^2)^2}\p_xu_2, u_2 \ra dt \notag
\\
& \quad +\int_0^T \la e^{-sf} sf_x\arctan(x)(\p_x^2 +
A^{-2}(x)\p_\theta^2)u_2, u_2 \ra dt \notag \\
& \quad 
-\int_0^T \la e^{-sf} \arctan(x)2A'A^{-3}\p_\theta^2u_2, u_2 \ra dt
\notag \\
& \quad +\int_0^T \la [R, B])u_2, u_2 \ra dt \label{E:Q-comm-201}
\end{align}
term by term.  
As in the previous section we have
\[
\int_0^T |\la [R,B]u_2,u_2\ra| \,dt \le CT\|u_0\|_{H^{1/2}}^2, 
\]
so we move on to the next computation.  
Next we write
\begin{align*}
  \int_0^T &  \la -2e^{-sf}\la x \ra^{-2}\p_x^2 u_2,u_2\ra\,dt \\
  &= 2\int_0^T \la e^{-sf/2}\la x \ra^{-1}\p_xu_2,\p_x(e^{-sf/2}\la x \ra^{-1}u_2)\ra\,dt\\
&= 2\int_0^T \|e^{-sf/2}\la x \ra^{-1} \p_xu_2\|^2\,dt\\
&\quad + \int_0^T \la e^{-sf/2}\la x \ra^{-1}\p_xu_2, \p_x(e^{-sf/2}\la x \ra^{-1}) u_2\ra\,dt.
\end{align*}
Note that
\begin{align*}
&\int_0^T \left|\la e^{-sf/2}\la x \ra^{-1}\p_x u_2, \p_x(e^{-sf/2}\la x \ra^{-1})u_2\ra\right|\,dt\\
&\quad\le C\int_0^T \|\la D_x\ra^{-1/2} (e^{-sf/2}\la x \ra^{-1} \p_x u_2)\|\|\la D_x\ra^{1/2}[(\p_xe^{-sf/2}\la x \ra^{-1}] u_2\|\,dt\\
&\quad\le C\int_0^T \|u\|_{H^{1/2}}^2\,dt\\
&\quad\le CT\|u_0\|_{H^{1/2}}^2.
\end{align*}

The $\p_\theta^2$ term in \eqref{E:Q-comm-201} is taken care of similarly:
\begin{align}
\int_0^T &-\la e^{-sf}\arctan(x)2A'A^{-3}\p_\theta^2u_2,u_2\ra
\,dt\notag \\
&=\int_0^T \la \arctan(x)2A'A^{-3}\p_\theta
u_2,\p_\theta(e^{-sf}u_2)\ra \,dt \notag \\
&=\int_0^T \la e^{-sf}\arctan(x)2A'A^{-3}\p_\theta u_2,(-sf_\theta +
\p_\theta)u_2)\ra \,dt. \label{E:theta-2}
\end{align}
Then for the term from \eqref{E:theta-2} involving only one derivative we may again bound it from above by $C_T\|u_0\|_{H^{1/2}}^2$. The other term is
\[
\int_0^T\la e^{-sf}\arctan(x)2A'A^{-3}\p_\theta u_2,\p_\theta u_2\ra\,dt \ge c\int_0^T \||x|^{m}\la x \ra^{-m-3/2}\p_\theta u_2\|^2\,dt.
\]
The next term from \eqref{E:Q-comm-201} we estimate is
\[
\left| \int_0^T \la 2x\la x \ra^{-4}\p_xu_2,u_2\ra\,dt \right| \le C_T\|u_0\|_{H^{1/2}}^2,
\]
just as in the previous sections.

The remaining terms can be controlled by using our estimates  from the
previous section and again that $s$ is small and $| f_\theta| =
\O(|x|^{2m-1})$.
Collecting terms, we end up with $u_2$ satisfying the estimate 
\begin{align}
\int_0^T &\|\la x \ra^{-1}\p_x u_2\|^2_{L^2} + \||x|^m\la
x\ra^{m-3/2}\p_\theta u_2\|^2_{L^2}\,dt \notag \\
&\le C_T\|u_0\|_{H^{1/2}}^2 + C\int_0^T\|\la x \ra^{-2}D_\theta
\tilde{\psi}(D_x/\la D_\theta\ra)u\|_{L^2}^2\,dt. \label{E:u2-101}
\end{align}

Finally we make use of the micro-support property of $\psi(D_x/\la
D_\theta\ra)u$. This function cuts $u_2 = (1 - \psi(D_x/\la
D_\theta\ra)u$ off to where $\la D_\theta \ra \lesssim |\p_x|$, so we
have from the G{\aa}rding inequality
\[
\|\la x \ra^{-1} \la D_\theta\ra u_2\| \le C\|\la x \ra^{-1} \p_x
u_2\| + C \| u \|^2_{H^{1/2}}.
\]
Using this, we see that
\[
\int_0^T \|\la x \ra^{-1} \la D_\theta\ra u_2\|^2_{L^2}\,dt \le C_T\|u_0\|_{H^{1/2}}^2 + \int_0^T\|\la x \ra^{-2}D_\theta \tilde{\psi}(D_x/\la D_\theta\ra)u\|_{L^2}^2\,dt.
\]

Now let $\chi(x) \equiv 1$ near 0 and have compact support. Then
\begin{align*}
\int_0^T&\|\la x \ra^{-2}D_\theta \tilde{\psi}(D_x/\la D_\theta\ra)u\|_{L^2}^2\,dt \\
&\le C\int_0^T\|\la x \ra^{-2}D_\theta \tilde{\psi}(D_x/\la
D_\theta\ra)\chi(x)u\|_{L^2}^2
\,dt
\\
&\le C\int_0^T 
\|\la x \ra^{-2}D_\theta \tilde{\psi}(D_x/\la D_\theta\ra)(1-\chi(x))u\|_{L^2}^2\,dt 
\end{align*}
The second  term on the right hand side has integrand with support
away from $x = 0$, so can then be bounded using the estimate from the previous section:
\begin{align*}
\int_0^T \|\la x \ra^{-2}D_\theta \tilde{\psi}(D_x/\la D_\theta\ra)(1-\chi(x))u\|^2\,dt &\le C\int_0^T  \la x \ra^{-2}(1-\chi(x))D_\theta u\|^2\,dt\\
&\le \int_0^T \||x|^m\la x \ra^{m-3/2}\p_\theta u\|^2\,dt\\
&\le C_T\|u_0\|_{H^{1/2}}^2.
\end{align*}

Thus to finish this part of our estimate we need only bound
\[
\int_0^T\|\la x \ra^{-2}D_\theta \tilde{\psi}(D_x/\la D_\theta\ra)\chi(x)u\|_{L^2}^2\,dt,
\]
for $\tilde{\psi}$ with compact support, $\tilde{\psi} \equiv 1$ on the support of $\psi$.
Note that by estimating this with $\tilde{\psi}$, we will also bound
\[
\int_0^T \|\la x \ra^{-2} D_\theta \chi(x)u_1\|_{L^2}^2\,dt = \int_0^T
\|\la x \ra^{-2} D_\theta \chi(x)\psi(D_x / \la D_\theta \ra) u\|_{L^2}^2\,dt,
\]
which will then complete the local smoothing estimate. We begin this process in the next section.

\section{{High Frequency Estimate}}
We have proven our local smoothing estimate outside of a region that
is ``small'' in both space and frequency. This suggests that it will
be profitable to work microlocally. To that end, we wish to show that
estimating 
\[
\int_0^T \|\la x \ra^{-2}D_\theta \tilde{\psi}(D_x/\la D_\theta\ra) \chi(x)u\|_{L^2}^2\,dt
\]
follows from  proving an estimate of the form
\[
\|(Q+\tau)\psi\chi u\| \ge \|\la D_\theta\ra^{{r}}\psi\chi u\|
\]
for some $r \in \R$, where 
 $u$ is microlocalized near $(x,\xi/\la \eta\ra) = 0$. We do this by using a ``$TT^*$'' argument.

The operator to which we apply the argument will be $F(t)$. Define the operator $F(t)$ by
\[
F(t)g = \chi(x)\psi(D_x/\la D_\theta\ra) e^{-itQ}g(x,\theta).
\]
We need to determine for which values of $r$, with $0 \le r \le 1$,  we have a bounded map $F: L^2_{x}L^2_{\theta} \to L^2([0,T])L^2_{x}H^r_{\theta}$. 

We have
\[
F^*g = \int_0^T e^{i\tilde{t}Q}\psi(D_x/\la D_\theta\ra)\chi(x) \tilde{g}\,d\tilde{t}
\]
and we need to show
\[
F^*: L^2([0,T])L^2_{x}H^{-r}_{\theta} \to L^2_xL^2_{\theta}.
\]
Then
\[
FF^*\tilde{g} =  \chi(x)\psi(D_x/\la D_\theta\ra)\int_0^T e^{i(\tilde{t}-t)Q}\psi(D_x/\la D_\theta\ra)\chi(x) \tilde{g}\,d\tilde{t}
\]
and we need to show
\[
FF^*: L^2([0,T])L^2_{x}H^{-r}_{\theta} \to L^2([0,T])L^2_{x}H^{r}_{\theta}.
\]

We split this expression into two. Let
\[
v_1 =\int_0^t e^{i(\tilde{t}-t)Q}\psi(D_x/\la D_\theta\ra)\chi(x) \tilde{g}\,d\tilde{t}
\]
and
\[
v_2 = \int_t^T e^{i(\tilde{t}-t)Q}\psi(D_x/\la D_\theta\ra)\chi(x) \tilde{g}\,d\tilde{t}.
\]
Then
\[
FF^*\tilde{g} = \chi(x)\psi(D_x/\la D_\theta\ra)(v_1+v_2).
\]

We need to show
\[
\|\chi(x)\psi(D_x/\la D_\theta\ra)v_j\|_{L^2_tL^2_xH^{r}_{\theta}} \le C\|\tilde{g}\|_{L^2_tL^2_xH^{-r}_{\theta}}
\]
for $j = 1,2$, where we require some assumptions on $\tilde{g}$ which will be included in the statement of our theorem below.

Note that
\[
(D_t+Q)v_1 = -i\psi(D_x/\la D_\theta\ra)\chi(x)\tilde{g}
\]
and
\[
(D_t+Q)v_2 = i\psi(D_x/\la D_\theta\ra)\chi(x)\tilde{g}.
\]
Let $\hat{\cdot}$ denote the Fourier transform in time. Then
\[
(\tau+Q)\hat{v}_j = (-1)^{j}i\chi\psi\hat{\tilde{g}},
\]
If we can prove the bound
\[
\|\chi\psi \hat{v}_j\|_{L^2_\tau L^2_xH^{r}_{\theta}} \le C\|\tilde{g}\|_{L^2_\tau L^2_xH^{-r}_{\theta}},
\]
we will have shown that $FF^*: L^2_tL^2_xH_{\theta}^{-r} \to L^2_tL^2_xH_{\theta}^r$ (and thus $F: L^2_xL^2_\theta \to L^2_tL^2_xH^r_{\theta}$) is a bounded operator. To that end, we need to bound the operator
\[
\chi(x)\psi(D_x/\la D_\theta\ra)(Q+\tau)^{-1}\psi(D_x/\la D_\theta\ra)\chi(x)
\]
in the $L^2_xH^r_\theta \to L^2_xH^{-r}_\theta$ operator norm, uniformly in $\tau$.

This is equivalent to showing that there exists $C$ such that
\[
\|\la D_{\theta}\ra^{2r}u\|_{L^2_{x,\theta}} \le C\|(Q+\tau)u\|_{L^2_{x,\theta}}.
\]
Proving this estimate will be the subject of the next section. Proving it will bound $\chi(x)\psi(D_x/\la D_\theta\ra)e^{itQ}g$ in $L^2_tL^2_xH^r_{\theta}$, but we are ultimately interested in bounding it in $L^2_tL^2_xH^1_{\theta}$. To do so, we apply the bound to $\la D_\theta\ra^{1-r}g$, so we will ultimately end up with the bound
\[
\int_0^T \|\la D_\theta\ra\chi(x)\psi(D_x/\la D_\theta\ra)u\|^2_{L^2_{x,\theta}}\,dt \le C\|u_0\|^2_{H^{1-r}}.
\]

\section{{Microlocal  Resolvent Estimate}}
We state and prove the aforementioned resolvent estimate in order to finish the proof of the main theorem.

    \begin{theorem}\label{Main-theorem}
    Let $\epsilon > 0$ and let $p = \xi^2 + \eta^2 A^{-2}$. Suppose $f(x,\theta)$ is a compactly supported, smooth function such that
    \[
    |\p_x^j \p_\theta^k f| \le C|x|^{2m-1}
    \]
    for $x$ small and $j,k \le N$ for sufficiently large $N =
    N(m,\epsilon)$ and $j+k \ge 1$.  For $s$  sufficiently small, there exists $c > 0$ such that
    \[
    \|((e^{-sf}p)^w+\tau)u\|_{L^2_{x,\theta}} \ge c\|\la D_\theta\ra^{2/(m+1) - \epsilon}u\|_{L^2_{x,\theta}},
    \]
    for all $\tau$, provided that $u$ satisfies the following microlocal support properties: We require that $u$ be of the form
    \[
    u = b^w\tilde{u},
    \]
    where $b$ has (classical) symbol supported in the region where $|(x,\xi/\eta)| \le \delta/2$ for some  small $\delta$ and $|\eta| \ge M$ for  sufficiently large $M$.
    \end{theorem}
Broadly speaking, our proof uses a commutator argument. The basic structure is to make use of the fact that to highest order, the symbol of the commutator of two pseudodifferential operators is given by applying the Hamiltonian vector field of the one symbol to the other symbol. Recall that $Q = (e^{-sf}p)^w$ denotes the operator we are interested in. We define a symbol $a\in S^0_{\epsilon}$ such that $H_{e^{-sf}p} a$ has the required lower bound. Ignoring the issues of error terms coming from the pseudodifferential calculus for the moment, we will consider the quantity
\[
\la [Q + \tau, a^w]u,u\ra,
\]
where $a$ is yet to be determined, following \cite{ChWu-lsm}.


As we stated above, the lower bound makes use of the fact that to
highest order, $[Q, a^w] = i^{-1}(H_{e^{-sf}p}a)^w$, where $p =
\xi^2+A^{-2} \eta^2$. We seek $a$ such that the resulting symbol $H_{e^{-sf}q}a$ is of the form $|\eta|^{-\epsilon}(\xi^2+\eta^2x^{2m})$ at least where $x$ and $\xi\eta^{-1}$ are small and $\eta$ is large. We may then use a lower bound on this operator to achieve a lower bound of
\[
\|\la D_\theta\ra^{1/(m+1) - \epsilon/2}u\|^2 \le \left|\la a^wu,(Q + \tau)u\ra\right|.
\]


The operator we require a lower bound on is $(e^{-sf}q)^w+\tau$. To define the symbol of our commutant $a$, we first define
\[
\Lambda(t) = \int_0^t \la \tilde{t} \ra^{-1-\epsilon_0}\,d\tilde{t},
\]
where $\epsilon_0 > 0$ is a small fixed number so that the integral is
bounded.
The important facts about $\Lambda(t)$ is that it is a symbol of order 0, and $\Lambda(t) \sim t$ near 0.

We will also make use of cutoff functions $\chi(t)$ and
$\tilde\chi(t)$.  For $\delta>0$ small, let $\chi(t)$ be a smooth, compactly supported function such that $\chi(t) \equiv 1$ for $|t| \le \delta/2$ and $\chi(t) \equiv 0$ for $|t| \ge \delta$. Let $\tilde\chi(t)$ be a smooth function such that $\tilde\chi(t)\equiv 0$ for $|t| \le M$ and $\tilde\chi(t) \equiv 1$ for $|t| \ge 2M$.

Let
\[
a = \chi(x)\chi(\xi\eta^{-1})\Lambda(x)\Lambda(\xi|\eta|^{-\epsilon})\tilde\chi(\eta)
\]
and note that
\[
|\p_x^\alpha \p_\xi^\beta \p_\theta^\gamma\p_\eta^\delta a| \le C_{\alpha,\beta,\gamma,\delta}\la \xi\ra^{-\epsilon\beta} \la\eta\ra^{-\epsilon\delta},
\]
where we have used the fact that $\chi(\xi\eta^{-1})$ cuts off to where $|\xi| \le |\eta|$. Because of this inequality, $a \in S^0_{\epsilon}$.

\begin{theorem}\label{hamiltonian_theorem}
Let $p, a$ be as above. Then for any $\epsilon > 0$
there exists $c > 0$ such that
\[
\la e^{-sf}(H_{e^{-sf}p}a)^wu,u\ra \ge c\la (D_\theta)^{-\epsilon}(D_x^2 + D_\theta^2x^{2m})u,u\ra - \cO(\|\la D_\theta\ra^{-\epsilon/2})u\|^2).
\]
for all $u$ microsupported where $|x| \le \delta$, $|\xi\eta^{-1}| \le \delta$, and $|\eta| \ge M$ for some large $M$.
\end{theorem}
\begin{proof}
We begin by computing $H_{e^{-sf}p}a$. First recall that
\[
p(x,\xi,\theta,\eta) = \xi^2 + A^{-2}(x)\eta^2
\]
and
\[
a(x,\xi,\theta,\eta) = \chi(x)\chi(\xi\eta^{-1})\Lambda(x)\Lambda(\xi|\eta|^{-\epsilon})\tilde\chi(\eta).
\]
Note that $a$ does not depend on $\theta$, so no $(e^{-sf}p)_\eta
a_\theta$ term will appear in the symbol expansion for the
commutator. We next compute the necessary derivatives for the symbol
expansion of the commutator.  Recall that the notation for the $\eta$ derivative actually refers to a difference operator in $\eta$.
\begin{align*}
(e^{-sf}p)_x &= -2e^{-sf}A^{-3}A'\eta^2 - sf_xe^{-sf}p,\\
(e^{-sf}p)_\xi &= 2e^{-sf}\xi,\\
(e^{-sf}p)_{\theta} &= -sf_\theta e^{-sf}p\\
	a_x &= [\chi'(x)\chi(\xi\eta^{-1})\Lambda(x)\Lambda(\xi|\eta|^{-\epsilon})+ \chi(x)\chi(\xi\eta^{-1})\Lambda'(x)\Lambda(\xi|\eta|^{-\epsilon})]\tilde\chi(\eta),\\
	a_\xi &= [\eta^{-1}\chi(x)\chi'(\xi\eta^{-1})\Lambda(x)\Lambda(\xi|\eta|^{-\epsilon}) + |\eta|^{-\epsilon}\chi(x)\chi(\xi\eta^{-1})\Lambda(x)\Lambda'(\xi |\eta|^{-\epsilon})]\tilde\chi(\eta)\\
	a_\eta &= \Lambda(x)(\Lambda(\xi|(\eta+1)|^{-\epsilon}) - \Lambda(\xi|\eta|^{-\epsilon}))\chi(x)\chi(\xi\eta^{-1})\tilde\chi(\eta)\\
	&\quad + \Lambda(x)\Lambda(\xi|(\eta + 1)|^{-\epsilon})\chi(x)\left[\chi(\xi(\eta+1)^{-1})\tilde\chi(\eta+1) - \chi(\xi\eta^{-1})\tilde\chi(\eta)\right]
\end{align*}

Using this computation we write down $H_{e^{-sf}p}a$ and split it into two parts. We are only interested in the behavior of $H_{e^{-sf}p}a$ where $x$ is small, $\xi\eta^{-1}$ is small, and $\eta$ is large, so we separate out the terms of $H_{e^{-sf}p}a$ where derivatives or difference operators hit $\chi$. These terms are supported where $x$ is large, $\xi$ is large relative to $\eta$, or $\eta$ is small. Because we have already proven our local smoothing estimate in these regions, there is no need to apply our resolvent estimate there. 

We have
\begin{align*}
  H_{e^{-sf}p} & a \\
  &= \biggr[2e^{-sf}\xi \biggr(\chi'(x)\chi(\xi\eta^{-1})\Lambda(x)\Lambda(\xi|\eta|^{-\epsilon}) + \chi(x)\chi(\xi\eta^{-1})\Lambda'(x)\Lambda(\xi|\eta|^{-\epsilon})\biggr)\\
&\quad - (-2e^{-sf}A^{-3}A'\eta^2 - sf_xe^{-sf}p)\biggr( \chi(x)\chi'(\xi\eta^{-1})\eta^{-1}\Lambda(x)\Lambda(\xi|\eta|^{-\epsilon})\\
&\quad \quad+ |\eta|^{-\epsilon}\chi(x)\chi(\xi\eta^{-1})\Lambda(x)\Lambda'(\xi|\eta|^{-\epsilon})\biggr)\\
&\quad - (-sf_\theta e^{-sf}p)\Lambda(x) \biggr(\Lambda(\xi|(\eta+1|)^{-\epsilon}) - \Lambda(\xi|\eta|^{-\epsilon})\biggr)\chi(x)\chi(\xi\eta^{-1})\biggr](\tilde\chi(\eta))\\
&\quad +sf_\theta e^{-sf}p\Lambda(x)\Lambda(\xi|(\eta + 1)|^{-\epsilon})\chi(x)\left[\chi(\xi(\eta+1)^{-1})\tilde\chi(\eta+1) - \chi(\xi\eta^{-1})\tilde\chi(\eta)\right].
\end{align*}
We collect the terms involving derivatives of $\chi$ or $\tilde{\chi}$ and write
\begin{align*}
H_{e^{-sf}p}a &= \biggr[2\xi\Lambda'(x)\Lambda(\xi|\eta|^{-\epsilon}) +2A^{-3}A'|\eta|^{2-\epsilon}\Lambda(x)\Lambda'(\xi|\eta|^{-\epsilon})\\
&\quad + sf_xp|\eta|^{-\epsilon}\Lambda(x)\Lambda'(\xi|\eta|^{-\epsilon}) +sf_\theta p \Lambda(x) \biggr(\Lambda(\xi|(\eta+1)|^{-\epsilon}) - \Lambda(\xi|\eta|^{-\epsilon})\biggr)\biggr]\\
&\quad \times e^{-sf}\chi(x)\chi(\xi\eta^{-1})(\tilde\chi(\eta))\\
&\quad + r,
\end{align*}
where 
\[
\supp r \subset \{|x| \ge \delta/2\}\cup \{|\xi|\ge \delta|\eta|/2\} \cup \{|\eta| \le 2M\}.
\]

We use $g$ to denote the part of $H_{e^{-sf}p}a$ to which we devote most of our efforts. Let
\[
g = (2\xi\Lambda'(x)\Lambda(\xi|\eta|^{-\epsilon}) + 2|\eta|^{2-\epsilon}A^{-3}(x)A'(x)\Lambda (x)\Lambda'(\xi|\eta|^{-\epsilon}))e^{-sf}\chi(x)\chi(\xi\eta^{-1})\tilde\chi(\eta).
\]
We use $\tilde{g}$ to denote the terms in which derivatives have hit $e^{-sf}$:
\begin{align}\label{eq-ref-1}
\tilde{g} &= \biggr[ sf_xp|\eta|^{-\epsilon}\Lambda(x)\Lambda'(\xi|\eta|^{-\epsilon})\\\notag
&\quad +sf_\theta p\Lambda(x) \biggr(\Lambda(\xi|(\eta+1)|^{-\epsilon}) - \Lambda(\xi|\eta|^{-\epsilon})\biggr)\biggr]e^{-sf}\chi(x)\chi(\xi\eta^{-1})\tilde\chi(\eta),
\end{align}
so
\[
H_{e^{-sf}p}a = g + \tilde{g} + r.
\]

Our goal is, roughly, to show that $\tilde{g}$ can be absorbed into $g$ and that $g$ is bounded below by a small multiple of $|\eta|^{-\epsilon}(\xi^2+\eta^2x^{2m})$.

We begin by bounding $\tilde{g}$ from above. Because we will only apply this result to functions which are microlocally supported in the region where  $\chi(x) = 1$, $\chi(\xi\eta^{-1}) = 1$, and $\tilde\chi(\eta) = 1$, we omit the $\chi$ and $\tilde\chi$ factors. We start with the first term in (\ref{eq-ref-1}):
\begin{align*}
|sf_xp|\eta|^{-\epsilon}\Lambda(x)\Lambda'(\xi|\eta|^{-\epsilon})e^{-sf}|&\le C|sf_x|\eta|^{-\epsilon}(\xi^2+\eta^2A^{-2}(x))\Lambda(x)\Lambda'(\xi|\eta|^{-\epsilon})|\\
&\le C|sf_x|\eta|^{2-\epsilon}\Lambda(x)\Lambda'(\xi|\eta|^{-\epsilon})|,
\end{align*}
where we have used the fact that $|\xi| \le C|\eta|$.

For the next term in (\ref{eq-ref-1}) we first use the mean value theorem to note that
\begin{align*}
\left|\Lambda(\xi|(\eta+1)|^{-\epsilon}) - \Lambda(\xi|\eta|^{-\epsilon})\right| &= \left|\int_{\xi|\eta|^{-\epsilon}}^{\xi|(\eta+1)|^{-\epsilon}} \la t\ra^{-1-\epsilon_0}\,dt\right|\\
&\le C|\xi|\left||(\eta+1)|^{-\epsilon} - |\eta|^{-\epsilon}\right|\la \xi|\eta|^{-\epsilon}\ra^{-1-\epsilon_0}\\
&\le C|\xi|||\eta|^{-1-\epsilon}|\la \xi|\eta|^{-\epsilon}\ra^{-1-\epsilon_0}.
\end{align*}
Thus
\begin{align*}
\biggr|-sf_\theta p &\Lambda(x)\biggr(\Lambda(\xi|(\eta+1)|^{-\epsilon}) - \Lambda(\xi|\eta|^{-\epsilon})\biggr)\biggr]e^{-sf}\biggr| \\
&\le C\left|sf_\theta|\eta|^{-1-\epsilon}(\xi^2+\eta^2A^{-2}(x))\xi\Lambda(x)\right|\la \xi|\eta|^{-\epsilon}\ra^{-1-\epsilon_0}\\
&\le C|sf_\theta|\eta|^{2-\epsilon}\Lambda(x)\la \xi|\eta|^{-\epsilon}\ra^{-1-\epsilon_0}|,
\end{align*}
so 
\begin{equation}\label{pert-bound}
|\tilde{g}| \le |s|(|f_x| + |f_\theta|)||\eta|^{2-\epsilon}\Lambda(x)|\la \xi|\eta|^{-\epsilon}\ra^{-1-\epsilon_0}.
\end{equation}

We need to write $g$ in a more useful form. To get started, we recall that the definition of $\Lambda$ is
\[
\Lambda(t) = \int_0^t \la \tilde{t}\ra^{-1-\epsilon_0}\,d\tilde{t},
\]
so $\Lambda'(t) = \la t\ra^{-1-\epsilon_0}$, and
\begin{align*}
g & = (2\xi\la x\ra^{-1-\epsilon_0}\Lambda(\xi|\eta|^{-\epsilon}) +
2|\eta|^{2-\epsilon}A^{-3}(x)A'(x)\Lambda(
x)\la\xi|\eta|^{-\epsilon}\ra^{-1-\epsilon_0})\\
& \quad \times e^{-sf}\chi(x)\chi(\xi\eta^{-1})\tilde\chi(\eta).
\end{align*}
We will assume throughout that $|x| \le \delta/2$, $|\xi\eta^{-1}| \le \delta/2$, and $|\eta|\ge M$ because we will be applying our operators to functions microlocally cutoff near here. In this region, $\chi(x) = 1$, $\chi(\xi\eta^{-1}) = 1$, and $\tilde\chi(\eta) = 1$.

We first break $g$ up into two parts:
\begin{align*}
g &= (2\xi\la x\ra^{-1-\epsilon_0}\Lambda(\xi|\eta|^{-\epsilon}) + 2|\eta|^{2-\epsilon}A^{-3}(x)A'(x)\Lambda (x)\la\xi|\eta|^{-\epsilon}\ra^{-1-\epsilon_0})e^{-sf}\\
&= g_1 + g_2,
\end{align*}
where
\begin{align*}
g_1 &= 2\xi\la x\ra^{-1-\epsilon_0}\Lambda(\xi|\eta|^{-\epsilon})e^{-sf},\\
g_2 &= 2|\eta|^{2-\epsilon}A^{-3}(x)A'(x)\Lambda (x)\la\xi|\eta|^{-\epsilon}\ra^{-1-\epsilon_0}e^{-sf}.
\end{align*}

Before bounding $g$ from below, we note how $\tilde{g}$ may be absorbed into $g$. From (\ref{pert-bound}) we see that
\[
\tilde{g} \le C|s|g_2
\]
as long as $|f_x| \le CA'(x)$ and $|f_\theta| \le CA'(x)$. This is satisfied as long as $|f_x| \le C|x|^{2m-1}$ and $|f_\theta| \le C|x|^{2m-1}$.
We will consider two cases.

{\bf Case 1:}  In the first case we make the assumption that $|\xi|\eta|^{-\epsilon}| \le \delta$. We are working where $\Lambda$ is only applied to small quantities, and for $|t|$ small $\Lambda(t) = t + \cO(t^3)$ and $\la t\ra^{-1-\epsilon} = 1 + \cO(t^2)$.

We write out $g_1$. Because $|\eta|$ is relatively large and $\xi$ is relatively small, the most important term will end up being $2|\eta|^{-\epsilon}\xi^2e^{-sf}$. Below we will show how the other terms may be absorbed. We separate out this term by writing
\begin{align*}
g_1 &= (2\xi(1+\cO(x^2))(\xi|\eta|^{-\epsilon} + \cO((\xi|\eta|^{-\epsilon})^3)e^{-sf} \\
&= 2\xi^2|\eta|^{-\epsilon}(1+\cO(x^2))(1+\cO((\xi|\eta|^{-\epsilon})^2))e^{-sf} \\
&= 2|\eta|^{-\epsilon}\xi^2e^{-sf} + \cO(x^2\xi^2|\eta|^{-\epsilon}) + \cO(\xi^4|\eta|^{-3\epsilon}) + \cO(\xi^4|\eta|^{-3\epsilon}x^2)\\
&= 2|\eta|^{-\epsilon}\xi^2e^{-sf} + \xi^2|\eta|^{-\epsilon}\left(\cO(x^2) + \cO((\xi|\eta|^{-\epsilon})^2)\right)
\end{align*}
where we have used the fact that there exists $C$ such that $e^{-sf(x,\theta)} \le C$.

Because $|x| \le \delta$ and $|\xi|\eta|^{-\epsilon}| \le \delta$, we then have
\[
g_1 = 2|\eta|^{-\epsilon}\xi^2e^{-sf}(1+\cO(\delta^2)).
\]

Along the same lines, for $g_2$, the most important term in the expansion will be $2|\eta|^{-\epsilon}(|\eta| x^m)^2e^{-sf}$. Recall that $A(x) = (1+x^{2m})^{1/2m}$. Here we use Taylor's theorem to expand $A^{-3}(x)A'(x)$ and write
\begin{align*}
g_2 &= 2|\eta|^{2-\epsilon}A^{-3}(x)A'(x)(x+\cO(x^3))(1+\cO((\xi|\eta|^{-\epsilon})^2)))e
^{-sf}\\
&= 2|\eta|^{2-\epsilon}(x^{2m-1} + \cO(x^{4m-1}))x (1+\cO(x^2))(1+(\cO((\xi|\eta|^{-\epsilon})^2)))e^{-sf}\\
&= 2|\eta|^{-\epsilon}(|\eta| x^m)^2e^{-sf} + 2|\eta|^{2-\epsilon}x^{2m}\left((\cO(x^{4m}) + \cO((x^{m}\xi|\eta|^{-\epsilon})^2)\right).
\end{align*}

Again using that $|x| \le \delta$ and $|\xi|\eta|^{-\epsilon}| \le \delta$, we have
\[
g_2 = 2|\eta|^{-\epsilon}(|\eta| x^m)^2e^{-sf}(1 + \cO(\delta^2)).
\]
Because $|\tilde{g}| \le C|s|g_2$, we then have $g_2 + \tilde{g} = g_2(1+\cO(s))$, hence
\[
g_2 + \tilde{g} = 2|\eta|^{-\epsilon}(|\eta| x^m)^2e^{-sf}(1 + \cO(\delta^2) + \cO(s)).
\]
We can thus write
\[
g + \tilde{g} = 2e^{-sf}|\eta|^{-\epsilon}(\xi^2+|\eta|^2x^{2m})(1+\cO(\delta^2) + \cO(s))
\]
as long as $|\xi|\eta|^{-\epsilon}| \le \delta$.

{}
{\bf Case 2:}  
We move on to our other case, where $|\xi |\eta|^{-\epsilon}| \ge \delta$. Our cutoff functions still allow us to assume that $|x| \le \delta$, $|\xi\eta^{-1}| \le \delta$, and $|\eta|$ is large.  

In this region, we will show that $g + \tilde{g}$ is elliptic. We will consider two sub-cases, based on the size of $x$ relative to the size of $\xi|\eta|^{-\epsilon}$.

We first note that $g_1, g_2 \ge 0$. Also note using the bound on
$\tilde{g}$ given by (\ref{pert-bound}) that, for $s$ sufficiently
small, 
\begin{align*}
g_1 + g_2 + \tilde{g} = g_1 + g_2(1+\cO(s)) \ge g_1 + (1-C|s|)g_2 \ge c(g_1+g_2).
\end{align*}
Hence showing that $g$ is elliptic will show that $g + \tilde{g}$ is elliptic.

For our first sub-case, suppose $|x|^{1+\epsilon_0} \ge
|\xi|\eta|^{-\epsilon}| \ge \delta$.   Then 
\begin{align*}
g_2 &=  2e^{-sf}|\eta|^{2-\epsilon}A^{-3}(x)A'(x)\Lambda (x)\la\xi|\eta|^{-\epsilon}\ra^{-1-\epsilon_0}\\
&= 2e^{-sf}|\eta|^{2-\epsilon}x^{2m}(1+\cO(x^{2m}))(1+\cO(x^2))\la \xi|\eta|^{-\epsilon}\ra^{-1-\epsilon_0}\\
&\ge ce^{-sf}|\eta|^{2-\epsilon}x^{2m}(1+\cO(x^{2}))\la \xi|\eta|^{-\epsilon}\ra^{-1-\epsilon_0}\\
&\ge c|\eta|^{2-\epsilon}x^{2m}|\xi|\eta|^{-\epsilon}|^{-1-\epsilon_0}\\
&\ge c|\eta|^{2-\epsilon}x^{2m}|x|^{-(1+\epsilon_0)^2}\\
&\ge c|\eta|^{2-\epsilon}.
\end{align*}

For our second sub-case, if $|\xi|\eta|^{-\epsilon}| \ge |x|^{1+\epsilon_0}$, but still $|\xi|\eta|^{-\epsilon}| \ge \delta$, then
\begin{align*}
g_1 &= 2e^{-sf}\xi\la x\ra^{-1-\epsilon_0}\Lambda(\xi|\eta|^{-\epsilon})\\
& \geq c \frac{ \xi}{|\xi|} | \eta |^\epsilon \Lambda (\xi | \eta
|^{-\epsilon}) \\
&\ge c|\eta|^{\epsilon}
\end{align*}
since $\xi$ and $\Lambda (\xi |\eta |^{-\epsilon} )$ have the same
sign.  
Hence $g \ge c|\eta|^{\epsilon}$.
In either sub-case, we find that in this region $g \ge c|\eta|^{\epsilon}$ and hence $\tilde{g} + g \ge c|\eta|^{\epsilon}$.

Considering both cases, there thus exists a $\sigma > 0$ such that if
$u$ is microlocally supported only in the region where
$|\xi|\eta|^{-\epsilon}| \le \delta$, 
\[
\la (g+\tilde{g})^wu,u\ra \ge \sigma\|\la D_\theta\ra^{\epsilon/2}u\|^2.
\]
 We have shown that there exists $c > 0$ such that here we may write
\[
g+\tilde{g} \ge c|\eta|^{-\epsilon}(\xi^2+\eta^2x^{2m})(1+\cO(\delta^2)+\cO(s)).
\]
For $s>0$ sufficiently small, this then allows us to write
\[
g+\tilde{g} = |\eta|^{-\epsilon}(\xi^2+\eta^2x^{2m})K^2,
\]
where $K$ is a strictly positive symbol. Because we are using the Weyl quantization here, this quantizes as
\[
\Op^w(K)^*(D_\theta)^{-\epsilon}(D_x^2 + D_\theta^2x^{2m})\Op^w(K) + \cO(\la D_\theta\ra^{-\epsilon}),
\]
so that 
\[
\la (g+\tilde{g})^wu,u\ra \ge \la  (D_\theta)^{-\epsilon}(D_x^2 + D_\theta^2x^{2m})u,u\ra - \cO\|(\la D_\theta\ra^{-\epsilon/2})u\|^2
\]

We thus have
\[
\la (H_{e^{-sf}p}a)^wu,u\ra \ge \la  (D_\theta)^{-\epsilon}(D_x^2 + D_\theta^2x^{2m})u,u\ra - \cO(\|\la D_\theta\ra^{-\epsilon/2}u\|^2).
\]
\end{proof}

\begin{proof}[Proof of Theorem \ref{Main-theorem}]
In the symbol calculus, the commutator $[(e^{-sf}p)^w, a^w]$ has principal symbol $H_{e^{-sf}p}a$, but we will still need to control the remaining terms. Let
\[
R_1 = [(e^{-sf}p)^w,a^w] - (H_{e^{-sf}p}a)^w.
\]
Then we have
\[
[(e^{-sf}p)^w+\tau,a^w] = (H_{e^{-sf}p}a)^w + R_1.
\]
Applying this to $u$ and taking an inner product with $u$ we find the equality
\[
\la [(e^{-sf}p)^w+\tau,a^w]u,u\ra = \la (H_{e^{-sf}p}a)^wu,u\ra + \la R_1u,u\ra.
\]
We can then apply the above Theorem \ref{hamiltonian_theorem} to find
\begin{align}\label{med-eq-1}
\left|\la [(e^{-sf}p)^w+\tau,a^w]u,u\ra\right| &\ge c\la\la D_\theta\ra^{-\epsilon}(D_x^2+D_\theta^2x^{2m})u,u\ra - C\|\la D_\theta\ra^{-\epsilon/2}u\|^2\\\notag
&\quad - \left|\la R_1u,u\ra\right|
\end{align}

Our goal is to bound $\la R_1u,u\ra$ from above in such a way that it can be absorbed into the term $c\la \la D_\theta\ra^{-\epsilon}(D_x^2+D_\theta^2x^{2m})u,u\ra$.

The commutator $[(e^{-sf}p)^w, a^w]$ has symbol given by
\[
\left.\sum_{k=0}^N \frac{i^k}{k!}\sigma(D)^k\left[p(x,\xi,\theta,\eta) a(\tilde x , \tilde \xi, \tilde\theta, \tilde\eta) - (e^{-sf}p)(\tilde x, \tilde\xi, \tilde\theta, \tilde\eta) a(x,\xi,\theta,\eta)\right]\right|_{\text{diag}} + \cO(\la \eta \ra^{-\epsilon N}) ,
\]
where $\biggr|_{\text{diag}}$ denotes evaluation along the diagonal, i.e. $x = \tilde x$, $\xi = \tilde \xi$, $\theta = \tilde\theta$, and $\eta = \tilde\eta$. The bound on the error term is a result of the symbol class we are working in. Recall also from Theorem \ref{symb-1} that 
\begin{equation*}
A(D) = \frac{1}{2}\left(\la (D_\xi, D_\eta), (D_{\tilde{x}}, D_{\tilde\theta})\ra - \la (D_x, D_\theta),(D_{\tilde\xi}, D_{\tilde\eta}\ra\right).
\end{equation*}

The first non-zero term in this expansion is $H_{e^{-sf}p}a$. Because we are using the Weyl calculus, there are no even terms. Therefore, when applying this symbol expansion to write down $R_1$, the first term is
\begin{equation}
  \label{E:A-D}
  \left.\frac{i^3}{3!}A(D)^3\left[(e^{-sf}p)(x,\xi,\theta,\eta) a(\tilde x , \tilde \xi, \tilde\eta) - (e^{-sf}p)(\tilde x, \tilde\xi, \tilde\theta, \tilde\eta) a(x,\xi,\eta)\right]\right|_{\text{diag}}.\end{equation}

Before expanding $A(D)^3$, we note that $a$ does not depend on
$\theta$, so there will be no terms involving $\theta$ derivatives of
$a$, and thus no terms involving $\eta$ differences of $e^{-sf}p$. We
have to consider every combination of $x, \xi,$ and $\theta$ as the
derivative we will be applying to $e^{-sf}p$ in \eqref{E:A-D}. Because
of the symbol class of $a$ and $p$ we know that we will gain $3$
powers of $|\eta|^{-\epsilon}$. We also know that if any derivative
hits the term $e^{-sf}$ then we will have gained a derivative of $f$
and a factor of $s$. When this is the case, we can bound these  above by
\begin{equation}\label{remainder_factor}
C|sf_*|\eta|^{2-\epsilon}\Lambda(x)|,
\end{equation}
where $f_*$ denotes some (first, second, or third) derivative of $f$. As long as we require $|f_*| \le C|x|^{2m-1}$ we find that (\ref{remainder_factor}) is bounded above by $C|s|||\eta|^{2-\epsilon}|x|^{2m}$.

The remaining term occurs when three $x$ derivatives all hit $p$. This term is
\[
Ce^{-sf}(D_x^3 p)(D_\xi^3 a)= Ce^{-sf}(A^{-2})'''(x)|\eta|^{2-3\epsilon} \Lambda(x)\Lambda'''(\xi|\eta|^{-\epsilon})\chi(x)\chi(\xi\eta^{-1})(\tilde\chi(\eta)) + r_2,
\]
where 
\[
\supp r_2 \subset \{|x| \ge \delta\}\cup \{|\xi|\ge \delta|\eta|\} \cup \{|\eta| \ge \delta\}.
\]
Because we will be applying our estimate only on functions microlocally supported away from the support of $r_2$, it will pose no problem to absorb this term. For now, we simply carry this term along.

We first note that
\[
|(A^{-2})'''| \le C|x|^{2m-3}.
\]
Next note that
\[
|\Lambda(x)| \le |x|.
\]
Hence
\[
 C|(A^{-2})'''(x)|\eta|^{2-3\epsilon} \Lambda(x)\Lambda'''(\xi|\eta|^{-\epsilon})\chi(x)\chi(\xi\eta^{-1})(\tilde\chi(\eta)) |\le Cx^{2m-2}|\eta|^{2-3\epsilon}.
\]
Next we need to control the symbol $Cx^{2m-2}|\eta|^{2-3\epsilon}$.

When $|x|^{-1} \le c_0||\eta||^{\epsilon}$, we have
\[
|Cx^{2m-2}|\eta|^{2-3\epsilon}| \le Cc_0x^{2m}|\eta|^{2-\epsilon},
\]
which can be absorbed into \eqref{med-eq-1} as long as $c_0$ is small enough.  

On the other hand, if $|x| \le (c_0)^{-1}|\eta|^{-\epsilon}$ then
\[
|Cx^{2m-2}|\eta|^{2-3\epsilon}| \le C(c_0)^{-2m+2}|\eta|^{(1-2m)\epsilon}.
\]
Because $m \ge 2$ and $|\eta|$ is large, $|\eta|^{(1-2m)}$ is small,
so we may then absorb this term into $\|\la D_\theta\ra^{-\epsilon/2}u\|_{L^2}^2$. 

We may similarly bound $A(D)^k pa$ for higher powers. Note that for every further term we write down in the expansion, the power of $|\eta|$ in the remainder term is improved. This follows from the symbol class of $e^{-sf}p$ and $a$.

Before bounding the remainder term in this symbol expansion, we make a couple of notes. We know that $e^{-sf}p$ satisfies the inequalities
\[
|\p_x^\alpha \p_\xi^\beta \p_\theta^\gamma\p_\eta^\delta (e^{-sf}p)| \le C_{\alpha,\beta,\gamma,\delta} \la\eta\ra^{2-\epsilon(\beta+\delta)}.
\]

As we stated before, $a$ satisfies the inequalities for $\la \xi \ra
\lesssim \la \eta \ra$
\[
|\p_x^\alpha \p_\xi^\beta \p_\theta^\gamma\p_\eta^\delta a| \le C_{\alpha,\beta,\gamma,\delta} \la\eta\ra^{-\epsilon(\beta+\delta)}.
\]
Let $E_N$ denote the remainder term obtained after expanding the first $N$ terms of the commutator $[P, a^w]$. Using our hybrid calculus, we know that $E_N$ has symbol in the class $S^{2-N\epsilon}_{\epsilon}$, hence
\[
\|E_Nu\|_{L^2_{x,\theta}} \le \|\la D_\theta\ra^{2-\epsilon N} u\|_{L^2_{x,\theta}}.
\]
By taking $N$ large enough we will be able to absorb this term into our final lower bound.

Combining all this, we have shown
\[
|\la e^{-sf}R_1u,u\ra| \le C\|u\|^2 + C(c_0+|s|)\|(\la D_\theta\ra^{1-\epsilon/2}x^m)u\|^2,
\]
where $c_0$ is small. Note that
\begin{align*}
& c\la \la D_\theta\ra^{-\epsilon}(D_x^2+D_\theta^2x^{2m})u,u\ra -
  C(c_0+|s|)\|(\la D_\theta\ra^{1-\epsilon/2}x^m)u\|^2 \\
  & \quad \ge \tilde{c}\la \la D_\theta\ra^{-\epsilon}(D_x^2+D_\theta^2x^{2m})u,u\ra,
\end{align*}
where $\tilde{c} > 0$ is smaller than $c$.

In total we have found
\[
\left|\la [((e^{-sf}p)^w+\tau),a^w]u\ra\right| \ge c\la \la D_\theta\ra^{-\epsilon}(D_x^2+D_\theta^2x^{2m})u,u\ra - C\|u\|^2.
\]

We will now bound the left hand side from above. We expand the commutator to find
\[
\left|\la[((e^{-sf}p)^w+\tau), a^w]u,u\ra\right|  \le \left|\la ((e^{-sf}p)^w + \tau)a^wu,u\ra\right| + \left|\la a^w((e^{-sf}p)^w+\tau)u,u\ra\right|.
\]
Because $a^w$ and $(e^{-sf}p)^w + \tau$ are self-adjoint, we can combine these terms and obtain the bound
\[
\left|\la[((e^{-sf}p)^w+\tau), a^w]u,u\ra\right|  \le 2\left|\la ((e^{-sf}p)^w + \tau)u,a^wu\ra\right|.
\]

We now apply the identity operator, in the form $\la D_\theta\ra^{1/(m+1)-\epsilon/2}\la D_\theta\ra^{-1/(m+1)+\epsilon/2}$, to the right hand side:
\begin{align*}
  & \left|\la[((e^{-sf}p)^w+\tau), a^w]u,u\ra\right|  \\
  &\le 2\left|\la \la D_\theta\ra^{-1/(m+1)+\epsilon/2} ((e^{-sf}p)^w + \tau)u,\la D_\theta\ra^{1/(m+1)-\epsilon/2}a^wu\ra\right|\\
&\le C\|\la D_\theta\ra^{-1/(m+1)+\epsilon/2} ((e^{-sf}p)^w + \tau)u\|_{L^2}\|\la D_\theta\ra^{1/(m+1)-\epsilon/2}a^wu\|_{L^2}
\end{align*}
We compute: 
\begin{align}\notag
&C\|\la D_\theta\ra^{-1/(m+1)+\epsilon/2}((e^{-sf}p)^w + \tau)u\|_{L^2}\|\la D_\theta\ra^{1/(m+1)-\epsilon/2}a^wu\|_{L^2}\\\notag
&\le C\big(\|((e^{-sf}p)^w + \tau)\la
  D_\theta\ra^{-1/(m+1)+\epsilon/2}u\|_{L^2}^2 \\
  & \quad \notag + \|[\la D_\theta\ra^{-1/(m+1)+\epsilon/2},((e^{-sf}p)^w + \tau)]u\|_{L^2}^2\big)\\
&\quad \times \big(\|a^w\la D_\theta\ra^{1/(m+1)-\epsilon/2}u\|_{L^2}^2 + \|[\la D_\theta\ra^{1/(m+1)-\epsilon/2},a^w]u\|_{L^2}^2\big).\label{almost-final-bound}
\end{align}

To estimate \eqref{almost-final-bound} we note that
\[
\|[\la D_\theta\ra^{1/(m+1)-\epsilon/2},a^w]u\|_{L^2} \le C\|\la D_\theta\ra^{1/(m+1)-\epsilon/2}u\|_{L^2}.
\]
Furthermore if we expand the commutator $[\la D_\theta\ra^{-1/(m+1)+\epsilon/2}, (e^{-sf}p)^w+\tau)]$ using the symbol calculus and the condition on derivatives of $f$ we find that we can bound the first $N$ terms by
\[
C|s||x|^m\la \eta\ra^{1-\epsilon/2},
\]
and thanks to the gains in powers of $|\eta|^{-\epsilon}$ we can guarantee that the remainder term has bounded symbol.

Thus we can bound (\ref{almost-final-bound}) from above by
\begin{align*}
C&\left(\|((e^{-sf}p)^w + \tau)\la D_\theta\ra^{-1/(m+1)+\epsilon/2}u\| + |s|\|(\|C|x|^m\la D_\theta\ra^{1-\epsilon/2})^wu\| + \|u\|\right)\\
&\quad \times \|\la D_\theta\ra^{1/(m+1)-\epsilon/2}u\|\\
&\le \|((e^{-sf}p)^w + \tau)\la D_\theta\ra^{-1/(m+1)+\epsilon/2}u\|\|\la D_\theta\ra^{1/(m+1)-\epsilon/2}u\|\\
&\quad + c_1\|\la D_\theta\ra^{1/(m+1)-\epsilon/2}u\|^2 + Cc_1^{-1}\|u\|^2 + (Cc_1^{-1}|s|)\||x|^m\la D_\theta\ra^{1-\epsilon/2})^wu\|,
\end{align*}
where $c_1 > 0$ is very small.  Note similarly to before that
\begin{align*}
c&\la \la D_\theta\ra^{-\epsilon}(D_x^2+D_\theta^2x^{2m})u,u\ra-(Cc_1^{-1}|s|)\||x|^m\la D_\theta\ra^{1-\epsilon/2})^wu\|\\
&\quad \ge c_0\la \la D_\theta\ra^{-\epsilon}(D_x^2+D_\theta^2x^{2m})u,u\ra,
\end{align*}
where $c_0>0$ is slightly smaller than $c$, as long as $s$ is sufficiently small. 

Also note that since we are working microlocally where $|\eta|$ is large,
\[
\|u\|^2 \ll c_1\|\la D_\theta\ra^{1/(m+1)-\epsilon/2}u\|^2.
\]

We thus have the estimate
\begin{align*}
&\|((e^{-sf}p)^w + \tau)\la D_\theta\ra^{-1/(m+1)+\epsilon/2}u\|\|\la D_\theta\ra^{1/(m+1)-\epsilon/2}u\|\\
&\ge c\la \la D_\theta\ra^{-\epsilon}(D_x^2+D_\theta^2x^{2m})u,u\ra - 2c_1\|\la D_\theta\ra^{1/(m+1)-\epsilon/2}u\|^2. 
\end{align*}

Applying Lemma \ref{easy_lemma} we then find
\begin{align*}
&\|((e^{-sf}p)^w + \tau)\la D_\theta\ra^{-1/(m+1)+\epsilon/2}u\|\|\la D_\theta\ra^{1/(m+1)-\epsilon/2}u\|\\
&\ge c\|\la D_\theta\ra^{1/(m+1)-\epsilon/2}u\|^2 - 2c_1\|\la D_\theta\ra^{1/(m+1)-\epsilon/2}u\|^2\\
&\ge c\|\la D_\theta\ra^{1/(m+1)-\epsilon/2}u\|^2,
\end{align*}
as long as $c_1$ is sufficiently small. Dividing through by $\|\la D_\theta\ra^{1/(m+1)-\epsilon/2}u\|$ we have the inequality
\[
\|((e^{-sf}p)^w + \tau)\la D_\theta\ra^{-1/(m+1)+\epsilon/2}u\| \ge c\|\la D_\theta\ra^{1/(m+1)-\epsilon/2}u\|.
\]
Finally, plugging in $\la D_\theta\ra^{1/(m+1)-\epsilon/2}u$ we end up with the inequality
\[
\|((e^{-sf}p)^w + \tau)u\| \ge c\|\la D_\theta\ra^{2/(m+1)-\epsilon}u\|,
\]
which proves the theorem.
\end{proof}

The following Lemma and its proof follows Lemma A.2 in \cite{ChWu-lsm}.
\begin{lemma}\label{easy_lemma} There exists $c > 0$ such that
\[
\la \la D_\theta\ra^{-\epsilon}(-\p_x^2 - \p_{\theta}^2 x^{2m})u,u\ra \ge c\|\la D_\theta\ra^{1/(m+1) - \epsilon/2} u\|^2
\]
for all $u \in \mathcal{S}$ with microlocal support where $\eta > 0$.
\end{lemma}

This lemma depends on the following result on the anharmonic oscillator (See \cite{ReSi-IV}).
\begin{theorem} Let $P = -\p_x^2+x^{2m}$ with $m \in \Z_{\ge 2}$. Then as an operator on $L^2$ with domain $\mathcal{S}$, $P$ is essentially self-adjoint and has pure point spectrum with eigenvalues $\lambda_j \to \infty$. Every eigenfunction is in $\mathcal{S}$ and furthermore $\lambda_0 > 0$.
\end{theorem}

\begin{proof}[Proof of Lemma \ref{easy_lemma}]
Letting $\hat{u}$ denote the Fourier transform in only $\theta$, we note
\[
\la (-\p_x^2-\p_{\theta}^2x^{2m})u,u\ra = \la (-\p_x^2+\eta^2x^{2m}\hat{u},\hat{u}\ra.
\]
This inner product is
\[
\sum_{\eta \in \Z}\left(\int ((-\p_x^2+\eta^2x^{2m})\hat{u})\overline{\hat{u}}\,dx\right).
\]
We make the change of variables $x = |\eta|^{-1/(m+1)}\tilde{x}$ to obtain
\begin{align*}
\sum_{\eta \in \Z}|\eta|^{-1/(m+1)}&\left(\int ((-|\eta|^{2/(m+1)}\p_{\tilde{x}}^2+|\eta|^{2-2m/(m+1)}\tilde{x}^{2m})\hat{u})\overline{\hat{u}}\,d\tilde{x}\right)\\
&\ge c\sum_{\eta \in \Z}|\eta|^{1/(m+1)}\left(\int(-\p_{\tilde{x}}^2 + \tilde{x}^{2m})\hat{u}\overline{\hat{u}} \,d\tilde{x}\right).
\end{align*}
We then apply Lemma A.1 from \cite{ChWu-lsm} to find
\begin{align*}
\sum_{\eta \in \Z}|\eta|^{1/(m+1)}\left(\int(-\p_x^2 + \tilde{x}^{2m})\hat{u}\overline{\hat{u}} \,d\tilde{x}\right) &\ge c\sum_{\eta \in \Z}\int|\eta|^{1/(m+1)}\hat{u}\overline{\hat{u}}\,d\tilde{x}\\
&= c\sum_{\eta \in \Z}\int |\eta|^{2/(m+1)}\hat{u}\overline{\hat{u}}\,dx\\
&\ge c\la \la D_\theta\ra^{2/(m+1)}u,u\ra,
\end{align*}
where we have used the fact that this will only be applied to $\hat{u}$ supported away from $\eta = 0$.

To achieve the result with $\la D_\theta\ra^{-\epsilon}$ in front of
the operator, we apply the inequality we have just proven to the function $\la D_\theta\ra^{-\epsilon/2}$. Because $\la D_\theta\ra^{-\epsilon/2}$ is self-adjoint and commutes with the operator $\la D_\theta\ra^{-\epsilon}(-\p_x^2 - \p_\theta^2x^{2m})$, this proves the lemma.
\end{proof}

\bibliographystyle{alpha}
\bibliography{conf-bib}

\def\cprime{$'$} \def\cftil#1{\ifmmode\setbox7\hbox{$\accent"5E#1$}\else
  \setbox7\hbox{\accent"5E#1}\penalty 10000\relax\fi\raise 1\ht7
  \hbox{\lower1.15ex\hbox to 1\wd7{\hss\accent"7E\hss}}\penalty 10000
  \hskip-1\wd7\penalty 10000\box7}
\begin{thebibliography}{BCMP19}

\bibitem[BCMP19]{BCMP}
Robert Booth, Hans Christianson, Jason Metcalfe, and Jacob Perry.
\newblock Localized energy for wave equations with degenerate trapping.
\newblock {\em Math. Res. Lett.}, 26(4):991--1025, 2019.

\bibitem[Bur04]{Bur-sm}
N.~Burq.
\newblock Smoothing effect for {S}chr\"odinger boundary value problems.
\newblock {\em Duke Math. J.}, 123(2):403--427, 2004.

\bibitem[Chr07]{Chr-NC}
Hans Christianson.
\newblock Semiclassical non-concentration near hyperbolic orbits.
\newblock {\em J. Funct. Anal.}, 246(2):145--195, 2007.

\bibitem[Chr08]{Chr-disp-1}
Hans Christianson.
\newblock Dispersive estimates for manifolds with one trapped orbit.
\newblock {\em Comm. Partial Differential Equations}, 33:1147--1174, 2008.

\bibitem[Chr11]{Chr-QMNC}
Hans Christianson.
\newblock Quantum monodromy and non-concentration near a closed semi-hyperbolic
  orbit.
\newblock {\em Trans. Amer. Math. Soc.}, 363(7):3373--3438, 2011.

\bibitem[Chr18]{Chr-inf-deg}
Hans Christianson.
\newblock High-frequency resolvent estimates on asymptotically {E}uclidean
  warped products.
\newblock {\em Comm. Partial Differential Equations}, 43(9):1306--1362, 2018.

\bibitem[CM14]{ChMe-lsm}
Hans Christianson and Jason Metcalfe.
\newblock Sharp local smoothing for warped product manifolds with smooth
  inflection transmission.
\newblock {\em Indiana University Mathematics Journal}, pages 969--992, 2014.

\bibitem[CS88]{ConSau}
Peter Constantin and Jean-Claude Saut.
\newblock Local smoothing properties of dispersive equations.
\newblock {\em J. Amer. Math. Soc.}, 1(2):413--439, 1988.

\bibitem[CV71]{CV}
Alberto~P Calderon and R{\'e}mi Vaillancourt.
\newblock On the boundedness of pseudo-differential operators.
\newblock {\em Journal of the Mathematical Society of Japan}, 23(2):374--378,
  1971.

\bibitem[CW13]{ChWu-lsm}
Hans Christianson and Jared Wunsch.
\newblock Local smoothing for the schr{\"o}dinger equation with a prescribed
  loss.
\newblock {\em American Journal of Mathematics}, 135(6):1601--1632, 2013.

\bibitem[Dat09]{Dat-sm}
Kiril Datchev.
\newblock Local smoothing for scattering manifolds with hyperbolic trapped
  sets.
\newblock {\em Comm. Math. Phys.}, 286(3):837--850, 2009.

\bibitem[Doi96]{Doi}
Shin-ichi Doi.
\newblock Smoothing effects of {S}chr\"odinger evolution groups on {R}iemannian
  manifolds.
\newblock {\em Duke Math. J.}, 82(3):679--706, 1996.

\bibitem[H{\"o}r66]{HoHyp}
Lars H{\"o}rmander.
\newblock Pseudo-differential operators and hypoelliptic equations.
\newblock {\em Amer. Math. Soc. Symp. on Singular Integrals}, pages 138--183,
  1966.

\bibitem[Kat83]{Kato}
Tosio Kato.
\newblock On the {C}auchy problem for the (generalized) {K}orteweg-de {V}ries
  equation.
\newblock {\em Studies in Applied Mathematics: a volume dedicated to Irving
  Segal}, 8:93--128, 1983.

\bibitem[KY89]{KaYa-smooth}
Tosio Kato and Kenji Yajima.
\newblock Some examples of smooth operators and the associated smoothing
  effect.
\newblock {\em Rev. Math. Phys.}, 1(4):481--496, 1989.

\bibitem[RS78]{ReSi-IV}
Michael Reed and Barry Simon.
\newblock {\em Methods of modern mathematical physics. {IV}. {A}nalysis of
  operators}.
\newblock Academic Press [Harcourt Brace Jovanovich Publishers], New York,
  1978.

\bibitem[RT10]{PDO-on-torus}
Michael Ruzhansky and Ville Turunen.
\newblock Quantization of pseudo-differential operators on the torus.
\newblock {\em Journal of Fourier Analysis and Applications}, 16(6):943--982,
  2010.

\bibitem[Sj{\"o}87]{Sl}
Per Sj{\"o}lin.
\newblock Regularity of solutions to the {S}chrodinger equation.
\newblock {\em Duke Math. J.}, 55(3):699--715, 1987.

\bibitem[Tao06]{Tao-book}
Terence Tao.
\newblock {\em Nonlinear dispersive equations}, volume 106 of {\em CBMS
  Regional Conference Series in Mathematics}.
\newblock Published for the Conference Board of the Mathematical Sciences,
  Washington, DC, 2006.
\newblock Local and global analysis.

\bibitem[Tay81]{Tay-pdo}
Michael~E. Taylor.
\newblock {\em Pseudodifferential operators}, volume~34 of {\em Princeton
  Mathematical Series}.
\newblock Princeton University Press, Princeton, N.J., 1981.

\bibitem[Tay13]{Tay2}
Michael Taylor.
\newblock {\em Partial differential equations II: Qualitative studies of linear
  equations}, volume 116.
\newblock Springer Science \& Business Media, 2013.

\bibitem[Veg88]{Vega}
Luis Vega.
\newblock Schr\"odinger equations: pointwise convergence to the initial data.
\newblock {\em Proc. Amer. Math. Soc.}, 102(4):874--878, 1988.

\bibitem[Zwo12]{zw-book}
Maciej Zworski.
\newblock {\em Semiclassical Analysis}.
\newblock Graduate studies in mathematics. American Mathematical Society, 2012.

\end{thebibliography}

\end{document}